\documentclass[12pt]{article}
\usepackage{amsthm}
\usepackage{amsmath,amsfonts}
\usepackage{geometry}
\usepackage{amssymb}
\usepackage{kotex}
\usepackage{theoremref}
\usepackage{mathrsfs}
\usepackage{enumitem}

\newtheorem{thm}{Theorem}[section]
\newtheorem{lem}[thm]{Lemma}

\theoremstyle{definition}

\newtheorem{rmk}[thm]{Remark}
\newtheorem{eg}[thm]{Example}

\newcommand{\rn}{\mathbb{R}^n}

\newcommand{\supp}{\mathrm{supp\,}}

\newcommand{\hs}[3]{\dot{H}^{#1}_{#2, #3}}
\newcommand{\ihs}[3]{H^{#1}_{#2, #3}}
\newcommand{\scc}{\mathscr{S}}
\newcommand{\tpd}{\mathscr{S}'}
\newcommand{\ift}{^\vee}
\newcommand{\vp}{\varphi}

\newcommand{\wdh}{\,\widehat{\ }\,}
\newcommand{\rg}{\rangle}
\renewcommand{\lg}{\langle}
\renewcommand{\hat}{\widehat}

\newcommand{\hsf}[4]{\dot{F}^{#1,#4}_{#2,#3}}
\newcommand{\ihsf}[4]{F^{#1,#4}_{#2,#3}}
\newcommand{\hsb}[4]{\dot{B}^{#1,#4}_{#2,#3}}
\newcommand{\ihsb}[4]{B^{#1,#4}_{#2,#3}}
\newcommand{\lpo}[2]{\Delta_j^{#1} #2}
\newcommand{\inhomnorm}[6]{\left\| \left\{2^{j #1}\lpo{#5}{#6} \right\}_{j\in\mathbb{N}_0} \right\|_{L^{#2,#3}(l^{#4})}}
\newcommand{\homnorm}[6]{\left\|  \left\{ 2^{j #1}\lpo{#5}{#6} \right\}_{j\in\mathbb{Z}} \right\|_{L^{#2,#3}(l^{#4})}}
\newcommand{\inhomnormb}[6]{\left\| \left\{2^{j #1}\lpo{#5}{#6} \right\}_{j\in\mathbb{N}_0} \right\|_{l^{#4}(L^{#2,#3})}}
\newcommand{\homnormb}[6]{\left\|  \left \{ 2^{j #1}\lpo{#5}{#6} \right \}_{j\in\mathbb{Z}} \right\|_{l^{#4}(L^{#2,#3})}}
\newcommand{\les}{\lesssim}
\newcommand{\nhra}{\not\hookrightarrow}

\title{Interpolation inequalities in function spaces of Sobolev-Lorentz type}

\author{Jaeseong Byeon \thanks{%
Department of Mathematics, Sogang University, Seoul, 121-742, Korea (jsbyeon@sogang.ac.kr). }
\and Hyunseok Kim \thanks{%
Department of Mathematics, Sogang University, Seoul, 121-742, Korea (kimh@sogang.ac.kr).}
\and Jisu Oh \thanks{%
Department of Mathematics, Sogang University, Seoul, 121-742, Korea (jisuoh@sogang.ac.kr).
The research   was supported by Basic Science Research Program through the National Research Foundation of Korea (NRF) funded by the Ministry of Education (No. NRF-2016R1D1A1B02015245).}
}
\date{}

\begin{document}

\maketitle

\begin{abstract} Interpolation inequalities in    Triebel-Lizorkin-Lorentz spaces and Besov-Lorentz spaces are studied for both inhomogeneous and homogeneous cases. First we establish  interpolation inequalities under quite general assumptions on the parameters of the function spaces. Several results on necessary conditions are also provided. Next, utilizing  the interpolation inequalities together with some embedding results, we    prove   Gagliardo-Nirenberg inequalities for fractional derivatives in  Lorentz spaces, which do hold even  for the limiting case when one  of the parameters is equal to $1$ or $\infty$.


\end{abstract}

\section{Introduction}

By an interpolation inequality in quasi-Banach spaces $X_1 , X_2$, and $X$ with $X_1 \cap X_2 \subset X$, we mean   any inequality of the form
$$
\|f\|_X \les \|f\|_{X_1}^{1-\theta} \|f\|_{X_2}^\theta \quad  (f \in X_1 \cap X_2  ),
$$
where   $0< \theta < 1$ is a constant. Here  for two nonnegative quantities $a$ and $b$, we write $a\lesssim b$  if $a \le C b$ for some positive constant $C$.  If $a \les b $ and $b \les a $, we write $a \sim b$.

The celebrated Gagliardo-Nirenberg inequality    \cite{Ga,Niren} is an interpolation inequality in      Sobolev spaces of integral order   on the $n$ dimensional Euclidean space  ${\mathbb R}^n$.  Interpolation inequalities in more general function spaces have been obtained   for fractional Sobolev spaces \cite{BreMir}, for   Triebel-Lizorkin and Besov  spaces \cite{BreMir,Ozawa,Wa1},  for  Fourier-Herz spaces \cite{Chi}, for Soblev-Lorentz spaces \cite{dao,HaYuZh,McC1,Wa2}, and most recently   for Triebel-Lizorkin-Lorentz and Besov-Lorentz  spaces  \cite{WWY}.

For $s \in \mathbb R$, $1 \le p< \infty$, and $1 \le q, r \le \infty$, the Triebel-Lizorkin-Lorentz space $F_{p,q}^{s,r}$ on $\mathbb R^n$ is the generalization of the Triebel-Lizorkin space $F_p^{s,r}$ obtained by replacing the underlining Lebesgue space $L^p$ in the definition of $F_p^{s,r}$ by the Lorentz space $L^{p,q}$.
Sobolev-Lorentz spaces $H_{p,q}^s$ and Besov-Lorentz  spaces $B_{p,q}^{s,r}$ are defined in the same way. The corresponding homogeneous spaces can be also considered  and denoted by $\dot{F}_{p,q}^{s,r}$, $\dot{H}_{p,q}^s$, and $\dot{B}_{p,q}^{s,r}$, respectively.
All these function spaces of Sobolev-Lorentz type will be studied in   more details in the preliminary Section 2.

The main purpose of this paper is to establish general    interpolation inequalities in Triebel-Lizorkin-Lorentz spaces and Besov-Lorentz spaces for both inhomogeneous and homogeneous cases.

First of all, for  Triebel-Lizorkin-Lorentz spaces, we  consider the following interpolation inequalities of the  most general form:
\begin{equation}\label{itp-eq-inhomf}
\|f\|_{\ihsf{s}{p}{q}{r}}\lesssim \|f\|_{\ihsf{s_1}{p_1}{q_1}{r_1}}^{1-\theta}\|f\|_{\ihsf{s_2}{p_2}{q_2}{r_2}}^\theta ,
\end{equation}
\begin{equation}\label{itp-eq-homf}
\|f\|_{\hsf{s}{p}{q}{r}}\lesssim \|f\|_{\hsf{s_1}{p_1}{q_1}{r_1}}^{1-\theta}\|f\|_{\hsf{s_2}{p_2}{q_2}{r_2}}^\theta ,
\end{equation}
where    $s , s_1 , s_2 \in \mathbb{R}$, $1 \le p,p_1,p_2 < \infty$, $ 1 \le q,q_1,q_2,  r,r_1,r_2 \le \infty$, and $0<\theta<1$ are    numbers satisfying some conditions.

To state our results in a concise way, let $s_* , p_* , q_*$, and $r_*$ be   defined by
$$
s_* = (1-\theta )s_1 + \theta s_2 ,\quad \frac{1}{p_*}= \frac{1-\theta}{p_1}+\frac{\theta}{p_2} ,
$$
$$
\frac{1}{q_*}= \frac{1-\theta}{q_1}+\frac{\theta}{q_2} , \quad\mbox{and}\quad \frac{1}{r_*}= \frac{1-\theta}{r_1}+\frac{\theta}{r_2}.
$$
We first show  that    the   condition
\begin{equation}\label{nec-condition-common}
s_* -s \ge \frac{n}{p_*}-\frac{n}{p} \ge 0
\end{equation}
is  necessary and sufficient for the interpolation inequality  (\ref{itp-eq-inhomf}) to hold for $q  \ge  q_* $ and $r  \ge   r_*$.  See Theorems \ref{thm-itp-inhomf} and \ref {thm-necessity-00}  in Section 3 for precise statements.

We are more interested in finding sufficient and necessary conditions for  (\ref{itp-eq-inhomf}) to hold for $q   <   q_*  $ or  $r   < r_*  $. For the case when $r =1 < r_* =\infty$, we  prove  a rather complete result in  Theorems \ref{thm-itp-inhomf} and \ref{thm-necessity-0}:  if either one of the two conditions
\begin{enumerate}
\item[(a)] $s_* =s$, $p_*=p$, $s_1\neq s_2$
\item[(b)] $s_* >  s$, $s_*-s \ge n/p_*-n/p \ge 0$
\end{enumerate}
is satisfied, then (\ref{itp-eq-inhomf})   holds when  $q \ge q_*$ and   $r =1 < r_* =\infty$; conversely, if (\ref{itp-eq-inhomf})   holds when  $q = q_*$,   $r =1 < r_* =\infty$, and $1<p, p_1 , p_2 < \infty$, then  either    (a) or (b) must be  satisfied.
However, for the   case when $q=1< q_* = \infty$, we have only a partial result. Particularly, in Theorem \ref{thm-necessity}, we show that if (\ref{itp-eq-inhomf})   holds when  $q = 1< q_* =\infty$,   $r =1 < r_* =\infty$, and $1<p, p_1 , p_2 < \infty$, then  one of the following conditions is satisfied:
\begin{enumerate}[label = \textup{(\roman*)}]
\item $s_*  =s$, $p_*=p$, $s_1\neq s_2$, $p_1\neq p_2$, $s_2-s_1\neq n/p_2 - n/p_1$.
\item $s_* >  s$, $p_*=p$, $p_1 \neq p_2$.
\item $s_*-s=n/p_* - n/p>0$, $s_2-s_1\neq n/p_2 - n/p_1$.
\item $s_*-s>n/p_*-n/p>0$.
\end{enumerate}
Furthermore, it is shown in Theorem \ref{thm-itp-inhomf} that the conditions (ii), (iii), and (iv) are sufficient for the inequality (\ref{itp-eq-inhomf}) to  hold for   $q = 1< q_* =\infty$ and  $r =1 < r_* =\infty$. But we have been unable to prove sufficiency of (i) yet, which seems to be still open.

We next establish several results on sufficient and necessary conditions for the interpolation inequality (\ref{itp-eq-homf}) in homogeneous Triebel-Lizorkin-Lorentz spaces.  Among other things, we show in Theorem  \ref{thm-itp-homf}   that if the condition (iii) is satisfied and $1< p, p_1 , p_2 < \infty$, then (\ref{itp-eq-homf})  holds for   $q = 1< q_* =\infty$ and  $r =1 < r_* =\infty$, that is, there holds the following interpolation inequality:
\begin{equation}\label{inter-ineq-homo-TLL}
\|f\|_{\dot F^{s,1}_{p,1}}\les \|f\|_{\dot F^{s_1,\infty}_{p_1,\infty}}^{1-\theta}\|f\|_{\dot F^{s_2,\infty}_{p_2,\infty}}^\theta .
\end{equation}

For  inhomogeneous and homogeneous Besov-Lorentz spaces, we   consider the following  interpolation inequalities:
\begin{equation}\label{itp-eq-inhomb}
\|f\|_{\ihsb{s}{p}{q}{r}}\lesssim \|f\|_{\ihsb{s_1}{p_1}{q_1}{r_1}}^{1-\theta}\|f\|_{\ihsb{s_2}{p_2}{q_2}{r_2}}^\theta ,
\end{equation}
\begin{equation}\label{itp-eq-homb}
\|f\|_{\hsb{s}{p}{q}{r}}\lesssim \|f\|_{\hsb{s_1}{p_1}{q_1}{r_1}}^{1-\theta}\|f\|_{\hsb{s_2}{p_2}{q_2}{r_2}}^\theta ,
\end{equation}
where     $s , s_1 , s_2 \in \mathbb{R}$, $1 \le p,p_1,p_2  \le  \infty$, $ 1 \le q,q_1,q_2,  r,r_1,r_2 \le \infty$, and $0<\theta<1$. In Theorems \ref{thm-itp-inhomb} and \ref{thm-necessity-00-b}, we    show that the   condition (\ref{nec-condition-common})
is   necessary and sufficient for the interpolation inequality (\ref{itp-eq-inhomb}) to hold for $q  \ge  q_* $ and $r  \ge   r_*$. It will be also shown in Theorem \ref{thm-itp-inhomb}  that if one of the conditions (ii), (iii), and (iv) is satisfied, then the   interpolation inequality
\[
\|f\|_{\ihsb{s}{p}{1}{1}}\lesssim \|f\|_{\ihsb{s_1}{p_1}{\infty}{\infty}}^{1-\theta}\|f\|_{\ihsb{s_2}{p_2}{\infty}{\infty}}^\theta \]
holds even for the limiting case when one of $ p , p_1 $, and $p_2$ equals to $1$ or $ \infty$. The corresponding result for homogeneous spaces is guaranteed by Theorem \ref{thm-itp-homb}, from which we conclude that if   the condition (iii) is satisfied, then
\begin{equation}\label{inter-ineq-homo-BL}
\|f\|_{\dot B^{s,1}_{p,1}}\les \|f\|_{\dot B^{s_1,\infty}_{p_1,\infty}}^{1-\theta}\|f\|_{\dot B^{s_2,\infty}_{p_2,\infty}}^\theta .
\end{equation}
This result was already proved  by Wang et al. \cite{WWY} for the special    case when $1< p_1 , p_2 \le \infty$.

The second purpose of this paper is to extend         Gagliardo-Nirenberg inequalities to the setting  of Lorentz spaces by making   use of    the interpolation inequalities (\ref{itp-eq-homf}), (\ref{inter-ineq-homo-TLL}),  and (\ref{inter-ineq-homo-BL}) together with some embedding results.

By a standard argument, it can be shown (see Theorem \ref{SFE2-inhom}) that
$$
\|f\|_{\hsf{s}{p}{q}{2}} \sim  \|f\|_{\hs{s}{p}{q}} = \|\Lambda^s f\|_{L^{p,q}}
$$
for $s\in\mathbb{R}$, $1<p<\infty$, and $1\le  q \le \infty$, where $\Lambda^s = (-\Delta)^{s/2}$ is the fractional Laplacian of order $s$. Therefore, from  (\ref{inter-ineq-homo-TLL}), we can deduce that the   generalized       Gagliardo-Nirenberg inequality
\begin{equation}\label{inter-ineq-homo-SL}
\|\Lambda^s f\|_{L^{p,q}}\les \|\Lambda^{s_1} f\|_{L^{p_1,q_1}}^{1-\theta}\|\Lambda^{s_2} f\|_{L^{p_2,\infty}}^\theta
\end{equation}
holds for $1 \le q , q_1\le \infty$, if $s, s_1 , s_2 \in \mathbb R$, $1< p, p_1 , p_2 < \infty$, and $0<\theta < 1$ satisfy  the condition (iii).
Particularly, taking $s=s_1 =0$ in (\ref{inter-ineq-homo-SL}), we have
\begin{equation}\label{inter-ineq-homo-SL-00}
\|f\|_{L^{p,q}}\les \|f\|_{L^{p_1,q_1}}^{1-\theta}\|\Lambda^{s_2} f\|_{L^{p_2,\infty}}^\theta ,
\end{equation}
provided that  $s_2 > 0$, $1 <  p  , p_1 , p_2 <  \infty$,  and $0<\theta<1$ satisfy
\begin{equation}\label{cond-GGNI}
\frac{1}{p}=\frac{1-\theta}{p_1}+\theta\left(\frac{1}{p_2}-\frac{s_2}{n}\right)
\quad\text{and}\quad\frac{1}{p_1} \neq  \frac{1}{p_2}-\frac{s_2}{n}.
\end{equation}
Furthermore, it will be shown in Theorem  \ref{cor-BL1}  that (\ref{inter-ineq-homo-SL-00}) holds even for the limiting case when $p=q=\infty$ or   $1 \le p_1 =q_1 \le \infty$,  if $s_2 > 0$, $1  <  p  \le \infty$, $1 \le p_1 \le \infty$, $1< p_2 <  \infty$,  and $0<\theta<1$ satisfy (\ref{cond-GGNI}). Several useful inequalities can be  derived from (\ref{inter-ineq-homo-SL-00}); for instance, if $1 \le p_1 <p < \infty$, $1<p_2 < \infty$, and $s_2 =n/p_2 $, then
\[
\|f\|_{L^{p,1}} \lesssim      \|f\|_{L^{p_1, q_1}}^{p_1 /p}\|\Lambda^{n/p_2} f\|_{L^{p_2 ,\infty}}^{1-p_1/p}  ,
\]
where $q_1 =\infty$ if $p_1 >1$ and $q_1 =1$ if $p_1 =1$. This inequality holds even when $p_1 =1$ and $q_1=\infty$, by   Theorem  \ref{cor2}. More inequalities are provided in the examples following Theorems  \ref{cor-BL1} and  \ref{cor2}, which   generalize  recent results in \cite{dao,McC1,WWY} to some limiting cases.

The rest of the paper is organized as follows. In Section 2, we provide some preliminary results for  Triebel-Lizorkin-Lorentz spaces and Besov-Lorentz spaces. In Section 3, we state and prove our main results on sufficiency and necessity for interpolation inequalities in  Triebel-Lizorkin-Lorentz spaces.
Analogous results for Besov-Lorentz spaces are established in Section 4. Section 5 is devoted to deriving generalized Gagliardo-Nirenberg inequalities in Lorentz spaces from the interpolation inequalities  in Sections 3 and 4. Finally, proofs of some results in the previous sections are provided in Appendix.

\section{Preliminaries}
For two quasi-normed spaces $X$ and $Y$,  we say that $X$ is continuously embedded into $Y$ and write $X\hookrightarrow Y$ if   $X\subset Y$ and
$\|f\|_{Y} \les   \|f\|_{X}$ for all $f \in X$.
If $X$ is not continuously embedded into $Y$, we write $X\nhra Y$.
Furthermore, $X=Y$ always means that $X\hookrightarrow Y$ and $Y\hookrightarrow X$.

In this preliminary section, we define Lorentz spaces, Triebel-Lizorkin-Lorentz spaces and Besov-Lorentz spaces. Basic interpolation  and embedding results are then  introduced, with proofs of some results being postponed later in Appendix.

\subsection{Lorentz spaces}

For $1\le p < \infty$, $1\le q \le \infty$ or $p=q=\infty$, let   $L^{p,q}$  denote  the Lorentz space of  measurable functions $f$ on ${\mathbb R}^n$, which is a quasi-Banach space equipped with the quasi-norm
\begin{equation*}\label{lorentz-quasi-norm}
\|f \|_{L^{p,q}}=
\begin{cases}
 \displaystyle \left(p \int _0 ^\infty [\alpha \mu_f(\alpha)^{1/p}]^q\frac{d\alpha}{\alpha} \right )^{1/q}&\text{if}\,\, p,q<\infty\\
 	 \quad \displaystyle \sup_{\alpha>0} \alpha \mu_f(\alpha)^{1/p} &\text{if}\,\, p<\infty,\,q=\infty \\
	\inf \{\alpha>0: \mu_f(\alpha) =0\} &\text{if}\,\, p=q=\infty ,
 \end{cases}
\end{equation*}
where  $\mu_f$ denotes   the distribution function of $f$: $\mu_f(\alpha)= |\{x\in\mathbb{R}^n:|f(x)|>\alpha\}|$.

For notational convenience, we  define $L^{\infty,q}=\{0\}$ for $1 \le q<\infty$.
Note then that  $L^{p,q_1}\hookrightarrow L^{p,q_2}$ if and only if $1\le q_1\le q_2\le\infty$. Moreover, if $1 \le p=q \le \infty$, then $L^{p,q}=L^{p}$.
It is also well-known that if $1<p<\infty$, then the quasi-norm $\|\cdot\|_{L^{p,q}}$  is equivalent to a  norm  (see, e.g., \cite[Lemma 4.4.5]{Bennett}).
The following H\"older inequality in Lorentz spaces is due to O'Neil \cite[Theorem 3.4]{oneil}.

\begin{lem}\label{Holder-Lorentz}
Let $1\le p,p_1,p_2,q,q_1,q_2 \le \infty$ satisfy
$$
\frac{1}{p}=\frac{1}{p_1}+\frac{1}{p_2}\quad\text{and}\quad \frac{1}{q}\le \frac{1}{q_1}+\frac{1}{q_2}.
$$
Then for all $f\in L^{p_1,q_1}$ and $g\in L^{p_2,q_2}$,
$$
\|fg\|_{L^{p,q}} \les \|f\|_{L^{p_1,q_1}} \|g\|_{L^{p_2,q_2}}.
$$
\end{lem}

Let $\Gamma$ be an index set (typically, $\Gamma$ is the set $\mathbb{Z}$ of all integers or the set $\mathbb{N}_0$ of all nonnegative integers). For $1 \le r\le \infty$, we denote by $l^r=l^r(\Gamma)$ the Banach space of all complex-valued sequences $ \{a_j\}  =\{a_j\}_{j\in\Gamma}$ such that
\begin{equation*}
\|\{a_j\}\|_{l^r}=
\begin{cases}
\displaystyle \left(\sum_{j\in\Gamma} |a_j|^r\right)^{1/r} &\quad\text{if}\,\,r<\infty \\
\,\, \displaystyle \sup_{j\in\Gamma} |a_j|  &\quad\text{if}\,\,r=\infty
\end{cases}
\end{equation*}
is finite.
Note that $\| \{a_j\} \|_{l^{r_2}} \le \| \{a_j\} \|_{l^{r_1}}$ for $1\le r_1 \le r_2 \le \infty$.
Furthermore, if $ 1\le r,r_1,r_2 \le \infty$  satisfy $1/r \le 1/r_1 + 1/r_2$,  then
$\| \{ a_j  b_j  \} \| _{l^r} \le \| \{a_j  \}\| _{l^{r_1}}\| \{ b_j \} \| _{l^{r_2}}.$

For $1 \le p,q,r\le \infty$, let $L^{p,q}(l^r)=L^{p,q}(l^r(\Gamma))$ be the set of all sequences $\{f_j\}_{j\in\Gamma}$ of measurable functions on $\rn$ such that
\begin{equation*}
\|\{f_j\}_{j\in\Gamma}\|_{L^{p,q}( l^r)} =
  \left\| \| \{ f_j \}_{j\in\Gamma} \|_{l^r}    \right \|_{L^{p,q}} < \infty .
\end{equation*}
Similarly,
$l^r(L^{p,q})=l^r(\Gamma;L^{p,q})$ is the set of all sequences $\{f_j\}_{j\in\Gamma}$ of measurable functions on $\rn$ such that
\begin{equation*}
\|\{f_j\}_{j\in\Gamma}\|_{l^r(L^{p,q})} =
  \left\| \{\| f_j \|_{ L^{p,q}}\}_{j\in\Gamma}    \right \|_{l^r} < \infty .
\end{equation*}
Then $L^{p,q}(l^r)$ and $l^r(L^{p,q})$ are quasi-Banach spaces equipped with the quasi-norms above. Moreover,   each of the embeddings
\begin{equation}\label{eq-ebd}
L^{p,q_1}(l^{r_1})\hookrightarrow L^{p,q_2}(l^{r_2}) \quad\text{and}\quad l^{r_1}(L^{p,q_1}) \hookrightarrow l^{r_2}(L^{p,q_2})
\end{equation}
holds if and only if $1\le q_1\le q_2\le \infty$ and $1\le r_1\le r_2\le \infty$.
Note also that
$$
\| \{|f_j| ^\alpha \} \|_{L^{p/\alpha,q/\alpha}(l^{r/\alpha})} =  \| \{ f_j \}\|^\alpha _{L^{p ,q }(l^{r })}
$$
whenever $0 < \alpha \le \min(p,q,r)$.
We write $L^p(l^r)=L^{p,p}(l^r)$ and $l^r(L^{p,p})=l^r(L^p)$.
The following interpolation result is well-known (see, for instance, \cite[Theorem 5.3.1]{Bergh}).

\begin{lem}\label{interpolation-Lp}
Let $1 \le p, p_1 ,  p_2 \le \infty$ and $0<\theta<1$ satisfy
$$
p_1 \neq p_2 \quad \text{and}\quad \frac{1}{p}=\frac{1-\theta}{p_1}+\frac{\theta}{p_2}.
$$
Then for $1\le q,q_1,q_2,r\le\infty$ with $q_i=\infty$ when $p_i=\infty$ for $i=1,2$,
$$
(L^{p_1,q_1}(l^r),L^{p_2,q_2}(l^r))_{\theta,q}=L^{p,q}(l^r).
$$
\end{lem}

\subsection{Triebel-Lizorkin-Lorentz spaces}

We denote the Schwartz class on ${\mathbb R}^n$ by $\scc$ and its topological dual space by $\scc'$.

Let $\psi=\{\psi_j\}_{j\in\mathbb{N}_0}$ be a  sequence  in $C_c^{\infty}(\mathbb{R}^n)$ such that
\begin{equation}\label{cond1-inhom}
\supp \psi_0 \subset \{|\xi|\le2\},\quad \supp \psi_j\subset \{2^{j-1}\le |\xi|\le 2^{j+1}\}\;\;\text{for} \;j\in\mathbb{N},
\end{equation}
\begin{equation}\label{cond2-inhom}
\sup_{j\in\mathbb{N}_0} \sup_{\xi\in\mathbb{R}^n} 2^{j|\alpha|}|D^\alpha \psi_j(\xi)|<\infty\quad\text{for every multi-index} \,\, \alpha ,
\end{equation}
and
\begin{equation}\label{cond3-inhom}
\sum_{j=0}^\infty \psi_j(\xi)=1 \quad\mbox{for every}\; \xi\in\mathbb{R}^n.
\end{equation}
Associated with $\psi=\{\psi_j\}_{j\in\mathbb{N}_0}$, let $\{\Delta_j^\psi\}_{j\in\mathbb{N}_0}$ be  the sequence of Littlewood-Paley operators defined via the Fourier transform by
$$\Delta^\psi_jf=\left(\psi_j \hat f\,\right)\ift =\psi_j\ift \ast f\quad \text{for $f\in\scc'$.}$$
For $s\in\mathbb{R}$, $1\le p <\infty$, and $1\le  q,r \le \infty$,
we denote by $\ihsf{s}{p}{q}{r}$ the space of all $f \in \scc'$ such that
\begin{equation*}
\|f \|_{\ihsf{s}{p}{q}{r}}  = \inhomnorm{s}{p}{q}{r}{\psi}{f} < \infty .
\end{equation*}
Then  $\ihsf{s}{p}{q}{r}$  is a quasi-Banach space and will be called a Triebel-Lizorkin-Lorentz space.
If $1 \le p=q< \infty$, then $\ihsf{s}{p}{q}{r}$ coincides with the more familiar  Triebel-Lizorkin space $F_{p}^{s,r}$, which has been extensively studied by Triebel in \cite{Tri1,Tri4,Tri2,Tri3}.

To define homogeneous spaces,   let $\scc_0$ be a closed subspace of $\scc$ defined by
$$
\scc_0 = \{ \psi \in \scc \, :\, D^\alpha \hat \psi(0) =0\,\,\,\text{for all multi-indices}\,\, \alpha \}.
$$
The dual of $\scc_0$ is denoted by $\scc_0'$.
The following is taken from \cite[Proposition 2.4]{Tri3}.

\begin{lem}\label{uniqueness}
For any $f\in \scc_0'$, there exists $F\in\scc'$ such that
$\lg F,\psi \rg = \lg f , \psi \rg$  for all $\psi\in\scc_0$.
Moreover, if there is another $G\in\scc'$ such that $\lg G, \psi \rg =\lg f, \psi \rg$ for all $\psi\in\scc_0$, then $F-G$ is a polynomial.
\end{lem}

Let $\vp=\{\vp_j\}_{j\in\mathbb{Z}}$ be a sequence of functions in $C_c^\infty(\mathbb{R}^n)$ such that
\begin{equation}\label{cond1-hom}
\supp \vp_j\subset \{2^{j-1}\le |\xi|\le 2^{j+1}\}\quad\text{for $j\in\mathbb{Z}$},
\end{equation}
\begin{equation}\label{cond2-hom}
\sup_{j\in\mathbb{Z}} \sup_{\xi\in\mathbb{R}^n} 2^{j|\alpha|}|D^\alpha \vp_j(\xi)|<\infty\quad\text{for every multi-index $\alpha$},
\end{equation}
and
\begin{equation}\label{cond3-hom}
\sum_{j\in\mathbb{Z}} \vp_j(\xi)=1 \quad \mbox{for every}\,\, \xi\in\mathbb{R}^n\backslash\{0\}.
\end{equation}
As in the inhomogeneous case, let $\{\Delta_j^\vp\}_{j\in\mathbb{Z}}$ be the sequence of Littlewood-Paley operators associated with $\vp =\{\vp_j\}_{j\in\mathbb{Z}}$:
$$
\Delta_j^\vp f = \left(\vp_j \hat f \right)\ift =\vp_j\ift \ast f \quad \text{for $f\in\scc'$.}
$$
Note that $\lpo{\vp}{P}=0$ for every polynomial $P$.
Therefore, in view of  Lemma \ref{uniqueness}, for each $f\in\scc_0'$, we can define $\Delta_j^\vp f $ by
$$\Delta_j^\vp f=\Delta_j^\vp F,$$
where $F\in\scc'$ is any extension of $f$ to $\scc$. For $s\in \mathbb{R}$, $1 \le p<\infty$, and $1\le q,r \le \infty$, the homogeneous Triebel-Lizorkin-Lorentz space  $\hsf{s}{p}{q}{r}$ is defined as the space of all  $f \in \scc'_0$ such that
\begin{equation*}
\|f \|_{\hsf{s}{p}{q}{r}}  = \homnorm{s}{p}{q}{r}{\vp}{f} < \infty .
\end{equation*}
Then $\hsf{s}{p}{q}{r}$ is a quasi-Banach space. 
If $1 \le p=q<\infty$, $\hsf{s}{p}{q}{r}$ coincides with the homogeneous Triebel-Lizorkin space $\dot F^{s,r}_{p}$.

The following interpolation result was proved in   \cite[Theorem 2.4.2/1]{Tri2} for inhomogeneous Triebel-Lizorkin-Lorentz spaces $\ihsf{s}{p}{q}{r}$ with  $1<r < \infty$.

\begin{thm}\label{lem-itp-inhomf}
Let $1<p,p_1,p_2<\infty$ and $0<\theta<1$ satisfy
$$
p_1 \neq p_2 \quad \text{and}\quad \frac{1}{p}=\frac{1-\theta}{p_1}+\frac{\theta}{p_2}.
$$
Then for $s\in\mathbb{R}$ and  $1 \le q, r \le \infty$,
$$
( F ^{s,r}_{p_1},  F^{s,r}_{p_2})_{\theta,q} \hookrightarrow \ihsf{s}{p}{q}{r} \quad\mbox{and}\quad (\dot F ^{s,r}_{p_1}, \dot F^{s,r}_{p_2})_{\theta,q} \hookrightarrow \hsf{s}{p}{q}{r}.
$$
In addition, if $1<r \le \infty$,  then
$$
( F ^{s,r}_{p_1},  F^{s,r}_{p_2})_{\theta,q} = \ihsf{s}{p}{q}{r} \quad\mbox{and}\quad (\dot F ^{s,r}_{p_1}, \dot F^{s,r}_{p_2})_{\theta,q} = \hsf{s}{p}{q}{r}.
$$
\end{thm}

A proof of Theorem \ref{lem-itp-inhomf} will be provided later   in Appendix.

\subsection{Besov-Lorentz spaces}
Let $s\in\mathbb{R}$ and $1\le p,q,r\le \infty$. If $\psi=\{\psi_j\}_{j\in\mathbb{N}_0}$ is a  sequence  in $C_c^{\infty}(\mathbb{R}^n)$ that satisfies \eqref{cond1-inhom},  \eqref{cond2-inhom}, and  \eqref{cond3-inhom}, we define the Besov-Lorentz space $\ihsb{s}{p}{q}{r}$ by  the space of all $f \in \scc'$ such that
\begin{equation*}
\|f \|_{\ihsb{s}{p}{q}{r}}  = \inhomnormb{s}{p}{q}{r}{\psi}{f} < \infty.
\end{equation*}
The homogeneous Besov-Lorentz space $\hsb{s}{p}{q}{r}$ is defined by the space of all $f\in \scc_0'$ such that
\begin{equation*}
\|f \|_{\hsb{s}{p}{q}{r}}  = \homnormb{s}{p}{q}{r}{\vp}{f} < \infty ,
\end{equation*}
where $\vp=\{\vp_j\}_{j\in\mathbb{Z}}$ is a sequence of functions in $C_c^\infty(\mathbb{R}^n)$ satisfying \eqref{cond1-hom},  \eqref{cond2-hom}, and  \eqref{cond3-hom}. For $1 \le p=q \le \infty$, we write $B^{s,r}_p = \ihsb{s}{p}{p}{r}$ and $\dot B^{s,r}_p = \hsb{s}{p}{p}{r}$.

The following  embedding result is easily deduced  from the  definitions.

\begin{lem}\label{thm-ebd-FB and BF}
Let $s \in \mathbb{R}$ and $1 \le q \le \infty$.
\begin{enumerate}[label = \textup{(\roman*)}]
\item
If $1\le p <   \infty$, then $F^{s,\infty}_{p ,q} \hookrightarrow  B^{s, \infty}_{p,q}$ and
$\dot{F}^{s,\infty}_{p ,q} \hookrightarrow\dot B^{s, \infty}_{p,q}$.
\item If $1 < p < \infty$  or if $p=q=1$, then $B^{s,1}_{p ,q} \hookrightarrow  F^{s, 1}_{p,q}$
 and $\dot{B}^{s,1}_{p ,q} \hookrightarrow\dot F^{s, 1}_{p,q}$.
\end{enumerate}
\end{lem}

The interpolation space between two Besov spaces with the same $p$ is again a Besov space, as shown by the following result (see \cite[Section 6.4]{Bergh}, for example).

\begin{thm}\label{lem-itp-homb}
Let  $s, s_1,s_2\in\mathbb{R}$ and $0<\theta<1$ satisfy
$$s_1 \neq s_2 \quad\text{and}\quad
s=(1-\theta)s_1+\theta s_2.
$$
Then for $1\le p,r,r_1,r_2\le\infty$,
$$
( B ^{s_1,r_1}_{p},  B^{s_2,r_2}_{p})_{\theta,r} = B^{s,r}_p \quad \text{and}\quad ( \dot B ^{s_1,r_1}_{p}, \dot B^{s_2,r_2}_{p})_{\theta,r} =\dot B^{s,r}_p.
$$
\end{thm}

By Lemma \ref{thm-ebd-FB and BF} and Theorem \ref{lem-itp-homb}, we immediately obtain

\begin{thm}\label{lem-itp-homf-ihomf}
Let  $s, s_1,s_2\in\mathbb{R}$ and $0<\theta<1$ satisfy
$$s_1 \neq s_2 \quad\text{and}\quad
s=(1-\theta)s_1+\theta s_2.
$$
Then for $1 \le p< \infty$ and $1\le r,r_1,r_2\le\infty$,
$$
( F ^{s_1,r_1}_{p},  F^{s_2,r_2}_{p})_{\theta,r} = B^{s,r}_p \quad \text{and}\quad ( \dot F ^{s_1,r_1}_{p}, \dot F^{s_2,r_2}_{p})_{\theta,r} =\dot B^{s,r}_p.
$$
\end{thm}

\subsection{Some embedding results}

The following  is a list of  more or less standard embedding  results  for  the function spaces $\ihsf{s}{p}{q}{r}$, $\hsf{s}{p}{q}{r}$, $\ihsb{s}{p}{q}{r}$, and $\hsb{s}{p}{q}{r}$ for $s\ge0$, proofs of which will be  given   in Appendix.

\begin{thm}\label{prop-relation} Let $s>0$  and $1\le  q, r \le \infty$.
\begin{enumerate}[label = \textup{(\roman*)}]
\item
If  $1<p<\infty$, then $\ihsf{0}{p}{q}{2} = \hsf{0}{p}{q}{2}=L^{p,q} $.
\item  If $1<p<\infty$   or if $  p=q =1$, then
$$
F^{0,1}_{p,q}  \hookrightarrow L^{p,q}  ,   \quad  \dot{F}^{0,1}_{p,q} \hookrightarrow L^{p,q} , \quad\mbox{and}\quad \ihsf{s}{p}{q}{r}=L^{p,q}\cap \hsf{s}{p}{q}{r}.
$$
\item  If $1<p<\infty$   or if $1\le p=q \le \infty$, then
$$
 B^{0,1}_{p,q}  \hookrightarrow L^{p,q} \hookrightarrow B^{0,\infty}_{p,q} , \quad  \dot{B}^{0,1}_{p,q} \hookrightarrow L^{p,q} \hookrightarrow \dot B^{0,\infty}_{p,q} , \quad \mbox{and}\quad\ihsb{s}{p}{q}{r}=L^{p,q}\cap \hsb{s}{p}{q}{r}.
$$
\end{enumerate}
\end{thm}

The embeddings  between the Triebel-Lizorkin-Lorentz spaces and Besov-Lorentz spaces are completely characterized by Seeger and Trebels in \cite[Theorems 1.5 and 1.6]{seeger} as follows.

\begin{thm}\label{ebd-inhom-TLL}
Let $s_1,s_2\in\mathbb{R}$, $1 \le p_1,p_2<\infty$, and $1\le q_1,q_2,r_1,r_2\le \infty$. Then the embedding
$$
\ihsf{s_1}{p_1}{q_1}{r_1} \hookrightarrow \ihsf{s_2}{p_2}{q_2}{r_2}
$$
holds if and only if one of the following four conditions is satisfied:
\begin{enumerate}[label = \textup{(\roman*)}]
\item  $s_1=s_2$, $p_1=p_2$, $q_1\le q_2$, $r_1\le r_2$.
\item  $s_1 >s_2$, $p_1=p_2$, $q_1\le q_2$.
\item $s_1 -s_2 = n/p_1- n/p_2 >0 $, $q_1\le q_2$.
\item  $s_1 -s_2 > n/p_1- n/p_2 >0 $.
\end{enumerate}
\end{thm}

\begin{thm}\label{ebd-inhom-B}
Let $s_1,s_2\in\mathbb{R}$, $1 \le p_1,p_2<\infty$, and $1\le q_1,q_2,r_1,r_2\le \infty$. Then the embedding
$$
\ihsb{s_1}{p_1}{q_1}{r_1} \hookrightarrow \ihsb{s_2}{p_2}{q_2}{r_2}
$$
holds if and only if one of the following four conditions is satisfied:
\begin{enumerate}[label = \textup{(\roman*)}]
\item  $s_1=s_2$, $p_1=p_2$, $q_1\le q_2$, $r_1\le r_2$.
\item  $s_1 >s_2$, $p_1=p_2$, $q_1\le q_2$.
\item $s_1 -s_2 = n/p_1- n/p_2 >0 $, $r_1\le r_2$.
\item  $s_1 -s_2 > n/p_1- n/p_2 >0 $.
\end{enumerate}
\end{thm}

\begin{rmk}\label{rmk-ebd-inhom-B}
 In fact, it can be shown that $\ihsb{s_1}{p_1}{q_1}{r_1} \hookrightarrow \ihsb{s_2}{p_2}{q_2}{r_2}$ holds even when $p_1=q_1=\infty$ or $p_2=q_2= \infty$ if one of the conditions (i)-(iv) is satisfied. Indeed, it is easy to show  that if (i) or (ii) holds when $p_1 =p_2 =q_1 =q_2 =\infty$, then $B^{s_1,r_1}_{\infty}\hookrightarrow B^{s_2,r_2}_{\infty}$. Sufficiency of (iii) is proved by exactly the same argument as  in the proof of  Theorem \ref{ebd-hom-B} in Appendix. Finally, if (iv) holds when $p_1 < p_2 =q_2=\infty$, then $B^{s_1,r_1}_{p_1,q_1} \hookrightarrow B^{s_2  + n/p_1, r_2}_{p_1,q_1} \hookrightarrow B^{s_2,r_2}_{\infty}$ by sufficiency of (ii) and (iii).
\end{rmk}

It was remarked in \cite[p.1020]{seeger} that homogeneous counterparts of  Theorems \ref{ebd-inhom-TLL} and \ref{ebd-inhom-B} can be proved by the same arguments as in the inhomogeneous case, but without explicit statements. For the sake of completeness, we state and prove some  embedding theorems for homogeneous spaces.

The following    is due to Jawerth \cite[Theorem 2.1]{Jaw1}.

\begin{thm}\label{thm-ebd-FB-2}
Let $s_1,s_2\in\mathbb{R}$ and $1 \le   p_1,p_2 <  \infty$ satisfy
$$
s_1 - s_2 =\frac{n}{p_1}-\frac{n}{p_2}>0.
$$
Then
$$
\dot{F}^{s_1,\infty}_{p_1} \hookrightarrow\dot F^{s_2,1}_{p_2} , \quad \dot{F}^{s_1,\infty}_{p_1} \hookrightarrow  \dot B^{s_2,p_1}_{p_2}  , \quad\mbox{and}\quad \dot{B}^{s_1,r }_{p_1} \hookrightarrow\dot B^{s_2,r }_{p_2} \quad  (1 \le r \le \infty).
$$
\end{thm}

The proofs of the following theorems will be provided later in Appendix.

\begin{thm}\label{ebd-hom-TLL}
Let $s_1,s_2\in\mathbb{R}$, $1< p_1,p_2<\infty$, and $1\le q_1,q_2,r_1,r_2\le \infty$. Then the embedding
$$
\hsf{s_1}{p_1}{q_1}{r_1} \hookrightarrow \hsf{s_2}{p_2}{q_2}{r_2}
$$
holds if and only if one of the following conditions is satisfied:
\begin{enumerate}[label = \textup{(\roman*)}]
\item  $s_1=s_2$, $p_1=p_2$, $q_1\le q_2$, $r_1\le r_2$.
\item $s_1 -s_2 = n/p_1- n/p_2 >0 $, $q_1\le q_2$.
\end{enumerate}
\end{thm}

\begin{thm}\label{ebd-hom-B}
Let $s_1,s_2\in\mathbb{R}$, $1 \le p_1,p_2 \le \infty$, and $1\le q_1,q_2,r_1,r_2\le \infty$. Assume that $q_i=\infty$ when $p_i=\infty$ for $i=1,2$. Then the embedding
$$
\hsb{s_1}{p_1}{q_1}{r_1} \hookrightarrow\hsb{s_2}{p_2}{q_2}{r_2}
$$
holds if one of the following conditions is satisfied:
\begin{enumerate}[label = \textup{(\roman*)}]
\item  $s_1=s_2$, $p_1=p_2$, $q_1\le q_2$, $r_1\le r_2$.
\item $s_1 -s_2 = n/p_1- n/p_2 >0 $, $r_1\le r_2$.
\end{enumerate}
\end{thm}

\begin{rmk}
It is not difficult to show that   the embedding $\hsb{s_1}{p_1}{q_1}{r_1} \hookrightarrow\hsb{s_2}{p_2}{q_2}{r_2}$ holds only if $s_1 -s_2 = n/p_1- n/p_2 \ge 0 $ and $r_1\le r_2$ (see the proof of Theorem \ref{ebd-hom-TLL}).
\end{rmk}

Using  Theorems  \ref{ebd-hom-TLL} and  \ref{ebd-hom-B}, we can also prove   embedding results between Triebel-Lizorkin-Lorentz spaces and Besov-Lorentz spaces. 

\begin{thm}\label{thm-ebd-FB}
Let $s_1,s_2\in\mathbb{R}$ and $1 <   p_1,p_2 <  \infty$ satisfy
$$
s_1 - s_2 =\frac{n}{p_1}-\frac{n}{p_2}>0.
$$
Then for  $1 \le q , r \le \infty$,
$$
\dot{F}^{s_1,\infty}_{p_1 ,q} \hookrightarrow\dot B^{s_2,q}_{p_2,1} \quad\mbox{and}\quad
\dot{B}^{s_1,r}_{p_1,\infty} \hookrightarrow\dot F^{s_2,1}_{p_2 ,r}.
$$
\end{thm}

\begin{proof}
Let $\overline{p}_1$ and $\overline{p}_2$ be any numbers such that
$$
1<\overline{p}_1 <p_1 < \overline{p}_2 < p_2 \quad\mbox{and}\quad \frac{2}{p_1} =\frac{1}{\overline{p}_1} + \frac{1}{\overline{p}_2}.
$$
Then by  Theorem  \ref{lem-itp-inhomf}, Lemma \ref{thm-ebd-FB and BF}, Theorems     \ref{ebd-hom-B}, and \ref{lem-itp-homb}, we have
\begin{align*}
\hsf{s_1}{p_1}{q}{\infty}  =\left (\dot F^{s_1,\infty}_{\overline{p}_1}, \dot F^{s_1,\infty}_{\overline{p}_2}\right)_{1/2,q} &\hookrightarrow\left (\dot B^{s_1,\infty}_{\overline{p}_1}, \dot B^{s_1,\infty}_{\overline{p}_2}\right)_{1/2,q} \\
& \hookrightarrow \left  (\dot B^{s_1 -n/\overline{p}_1+n/\overline{p}_2 ,\infty}_{\overline{p}_2}, \dot B^{s_1,\infty}_{\overline{p}_2} \right)_{1/2,q}\\
&	 = \dot B^{s_1-n(1/\overline{p}_1-1/\overline{p}_2)/2, q}_{\overline{p}_2}\hookrightarrow \dot B^{s_2,q}_{p_2,1} .
\end{align*}
Similarly, if $p_3$ is chosen so that
$$
  \overline{p}_2 < p_2 <  {p}_3 < \infty\quad\mbox{and}\quad \frac{2}{p_2} =\frac{1}{\overline{p}_2} + \frac{1}{ {p}_3},
$$
then  by Theorems \ref{lem-itp-homf-ihomf} and  \ref{ebd-hom-TLL},
\begin{align*}
\dot{B}^{s_1,r}_{p_1,\infty}   \hookrightarrow \dot B^{s_2-n/p_2+n/\overline{p}_2,r}_{\overline{p}_2}
 &= \left  (\dot F^{s_2 ,1}_{\overline{p}_2}, \dot F^{s_2 - 2n/p_2 + 2n/\overline{p}_2,1}_{\overline{p}_2} \right)_{1/2,r} \\
 & \hookrightarrow
 \left (\dot F^{s_2,1}_{\overline{p}_2}, \dot F^{s_2,1}_{ {p}_3}\right)_{1/2,r}
	 \hookrightarrow \hsf{s_2}{p_2}{r}{1} . \tag*{\qedhere}
\end{align*}
\end{proof}

\section{Interpolation inequalities in Triebel-Lizorkin-Lorentz spaces}

Throughout this and next sections, let   $s , s_1 , s_2 \in \mathbb{R}$, $1 \le p,p_1,p_2 , q,q_1,q_2,  r,r_1,r_2 \le \infty$, and $0<\theta<1$ be fixed numbers. Assume in addition that $q=\infty$ if $p=\infty$ and that $q_i=\infty$ if $p_i=\infty$ for $i=1,2$.

\begin{rmk}\label{rmk-conditions}
The condition $s_*-s\ge n/p_* - n/p \ge 0$
 is necessary for each of the interpolation inequalities \eqref{itp-eq-inhomf} and  (\ref{itp-eq-inhomb}) to hold. To show this,  we adapt    the proof of \cite[Theorem 2.3.9]{Tri1}.  First we choose $\psi_0\in C_c^\infty(\rn)$ such that $\supp \psi_0 \subset \{|\xi|\le 3/2\}$ and $\psi_0=1$ on $\{|\xi|\le 1\}$. For $j \ge 1$, define $\psi_j(\xi)= \psi_0(2^{-j}\xi) - \psi_0(2^{-j+1}\xi)$. Then $\{\psi\}_{j\in\mathbb{N}_0}$ satisfies \eqref{cond1-inhom}, \eqref{cond2-inhom}, and \eqref{cond3-inhom}. Note that $\supp \psi_j \subset \{ 2^{j-1} \le |\xi|\le 3 \cdot 2^{j-1}\}$ and $\psi_j=1$ on $\{3\cdot 2^{j-2} \le |\xi| \le 2^j\}$ when $j\ge1$. Let $f\in\scc$ be chosen so that $\supp \hat{f}\subset \{a\le |\xi|\le b\}$ for some $3/4<a<b<1$. Then for each $k\in\mathbb{Z}$,
\begin{equation*}
\psi_j\widehat{f(2^k \cdot)}=
\begin{cases}
 \widehat{f(2^k \cdot)} &\text{if $k\le0, j=0$ or $k=j\ge1$,}\\
	0 &\text{otherwise.}
 \end{cases}
\end{equation*}
Thus for $s\in\mathbb{R}$, $1\le p\le\infty$, and $1\le q,r\le\infty$,
\begin{equation*}
\|f(2^k \cdot)\|_{B^{s,r}_{p,q}} =
\begin{cases}
2^{-kn/p}\|f\|_{L^{p,q}} &\text{if $k\le0$,}\\
2^{k(s-n/p)}\|f\|_{L^{p,q}} &\text{if $k\ge1$.}
 \end{cases}
\end{equation*}
In addition, if $1\le p<\infty$, then $\|f(2^k \cdot)\|_{F^{s,r}_{p,q}} =\|f(2^k \cdot)\|_{B^{s,r}_{p,q}}$.
Hence if \eqref{itp-eq-inhomf}  or (\ref{itp-eq-inhomb}) holds, then we must have
$$
2^{-kn/p} \les 2^{-kn/p_*} \quad \text{for $k\le 0$}
\quad\mbox{and}\quad
2^{k(s-n/p)} \les 2^{k(s_* - n/p_*)} \quad \text{for $k\ge1$},
$$
which imply that $s_*-s\ge n/p_* - n/p \ge 0$.
Similarly, it can be shown  that if \eqref{itp-eq-homf} or (\ref{itp-eq-homb}) holds, then $s_*-s = n/p_* - n/p \ge 0$ (see the proof of Theorem \ref{ebd-hom-TLL} in Appendix).
\end{rmk}

\begin{thm}\label{thm-itp-inhomf}
Assume that
\begin{equation*}
s_* -s \ge \frac{n}{p_*}-\frac{n}{p} \ge 0 .
\end{equation*}
Assume in addition that $ 1 \le p , p_1 , p_2 < \infty $.  Then the interpolation inequality
\[
\|f\|_{\ihsf{s}{p}{q}{r}}\lesssim \|f\|_{\ihsf{s_1}{p_1}{q_1}{r_1}}^{1-\theta}\|f\|_{\ihsf{s_2}{p_2}{q_2}{r_2}}^\theta
\]
holds for all $f \in\ihsf{s_1}{p_1}{q_1}{r_1} \cap \ihsf{s_2}{p_2}{q_2}{r_2}$, if one of the following conditions is satisfied:
\begin{enumerate}[label = \textup{(\roman*)}]
\item  $q_* \le q$, $r_* \le r$.
\item $ s_1=s_2$, $p_* = p$, $p_1 \neq p_2$, $\max (r_1 , r_2 ) \le r$.
\item    $s_1 \neq s_2$, $p_* = p$, $q_* \le q$.
\item $s_* >  s $,   $q_* \le q$.
\item $s_* >  s$, $p_*=p$, $p_1\neq p_2$.
\item  $s_* -s > n/p_* -n/p >0$.
\item  $s_* -s  = n/p_* -n/p >0$, $s_2 -s_1 \neq n/p_2 -n/p_1$.
\end{enumerate}
\end{thm}

To prove our main results, we need the following interpolation inequality in sequence spaces, taken from {\cite[Theorem 1.18.2]{Tri2}}.

\begin{lem}\label{itp-sum}
Assume that   $s_1 \neq s_2$.
Then for any complex sequence $\{a_j\}_{j\in\Gamma}$,
$$
\sum_{j\in\Gamma} 2^{js_*}|a_j| \les  \left(\sup_{j\in\Gamma} 2^{js_1}|a_j|\right)^{1-\theta}  \left( \sup_{j\in\Gamma} 2^{js_2}|a_j|\right)^{\theta}.
$$
\end{lem}

\begin{proof}[Proof of Theorem \ref{thm-itp-inhomf}]

Assume that $f\in \ihsf{s_1}{p_1}{q_1}{r_1} \cap \ihsf{s_2}{p_2}{q_2}{r_2}$. \vspace{0.2cm}

\noindent \emph{Case I.} Suppose that   $p_* =p$, $q_* \le q$,  and $r_* \le r$. Then noting that
$$
2^{js_*}|\lpo{\psi}{f}| = 2^{(1-\theta)js_1}|\lpo{\psi}{f}|^{1-\theta} 2^{\theta js_2}|\lpo{\psi}{f}|^\theta,
$$
we have
$$
\left \| \left\{ 2^{js_*}\Delta_j^\psi f \right\}_{j\in\mathbb{N}_0}\right \|_{l^r} \le \left \| \left\{2^{(1-\theta)js_1}|\Delta_j^\psi f |^{1-\theta}\right\}_{j\in\mathbb{N}_0}  \right \|_{l^{\frac{r_1}{1-\theta}}} \left \|  \left\{2^{\theta js_2}| \Delta_j^\psi f |^\theta \right\}_{j\in\mathbb{N}_0} \right \|_{l^{\frac{r_2}{\theta}}}.
$$
Hence by Theorem \ref{ebd-inhom-TLL} and Lemma \ref{Holder-Lorentz},
\begin{align*}
\|f\|_{\ihsf{s}{p}{q}{r}} & \les \|f\|_{\ihsf{s_*}{p}{q}{r}} = \inhomnorm{s_*}{p}{q}{r}{\psi}{f}\\
	&\les \left\| \left\{ 2^{(1-\theta) js_1}|\Delta_j^\psi f|^{1-\theta}\right\}_{j\in\mathbb{N}_0} \right\|_{L^{\frac{p_1}{1-\theta},\frac{q_1}{1-\theta}}(l^{\frac{r_1}{1-\theta}})} \left\| \left\{ 2^{\theta js_2}|\Delta_j^\psi f|^{\theta}\right\} _{j\in\mathbb{N}_0} \right\|_{L^{\frac{p_2}{\theta},\frac{q_2}{\theta}}(l^{\frac{r_2}{\theta}})}\\
	&= \inhomnorm{s_1}{p_1}{q_1}{r_1}{\psi}{f}^{1-\theta} \inhomnorm{s_2}{p_2}{q_2}{r_2}{\psi}{f}^\theta \\
	& = \|f\|^{1-\theta}_{\ihsf{s_1}{p_1}{q_1}{r_1}} \|f\|^\theta_{\ihsf{s_2}{p_2}{q_2}{r_2}}.
\end{align*}

\noindent \emph{Case II.} Suppose that $ s_1=s_2$, $p_* = p$, $p_1 \neq p_2$, and $\max (r_1 , r_2 ) \le r$. Then by Theorem \ref{ebd-inhom-TLL} and Lemma \ref{interpolation-Lp},
\begin{align*}
\|f\|_{\ihsf{s}{p}{q}{r}} & \les  \inhomnorm{s_*}{p}{q}{r}{\psi}{f} \\
	& \les \inhomnorm{s_*}{p_1}{q_1}{r}{\psi}{f}^{1-\theta} \inhomnorm{s_*}{p_2}{q_2}{r}{\psi}{f}^{\theta}\\
	& \le  \|f\|_{\ihsf{s_1}{p_1}{q_1}{r_1}} ^{1-\theta} \|f\|_{\ihsf{s_2}{p_2}{q_2}{r_2}}^\theta.
\end{align*}

\noindent \emph{Case III.} Suppose that $s_1 \neq s_2$, $p_* =p$, and $q_* \le q$.
Then by Theorem \ref{ebd-inhom-TLL}, Lemmas \ref{itp-sum} and  \ref{Holder-Lorentz},
\begin{align*}
\|f\|_{\ihsf{s}{p}{q}{r}} &\les \inhomnorm{s_*}{p}{q}{1}{\psi}{f}\\
	&\les \left\| \left(\sup_{j\in\mathbb{N}_0} 2^{js_1}|\lpo{\psi}{f}|\right)^{1-\theta} \left( \sup_{j\in\mathbb{N}_0} 2^{js_2}|\lpo{\psi}{f}|\right)^{\theta}\right\|_{L^{p,q}}\\
	&\les \left\| \left(\sup_{j\in\mathbb{N}_0} 2^{js_1}|\lpo{\psi}{f}|\right)^{1-\theta} \right\|_{L^{\frac{p_1}{1-\theta},\frac{q_1}{1-\theta}}}\left\|\left( \sup_{j\in\mathbb{N}_0} 2^{js_2}|\lpo{\psi}{f}|\right)^{\theta} \right\|_{L^{\frac{p_2}{\theta},\frac{q_2}{\theta}}}\\
	& \le \|f\|^{1-\theta}_{\ihsf{s_1}{p_1}{q_1}{r_1}} \|f\|^\theta_{\ihsf{s_2}{p_2}{q_2}{r_2}}.
\end{align*}

\noindent \emph{Case IV.} Suppose that $s_* >  s $ and    $q_* \le q$.
Then by Case I,
$$
\|f\|_{\ihsf{s_*}{p_*}{q}{\infty}} \les  \|f\|_{\ihsf{s_1}{p_1}{q_1}{r_1}}^{1-\theta}
\|f\|_{\ihsf{s_2}{p_2}{q_2}{r_2}}^\theta.
$$
Since $s_* >s$ and $s_* -s \ge n/p_* -n/p$, it follows from  Theorem \ref{ebd-inhom-TLL} that
\begin{align*}
\|f\|_{\ihsf{s}{p}{q}{r}}\les  \|f\|_{\ihsf{s_*}{p_*}{q}{\infty}} \les \|f\|_{\ihsf{s_1}{p_1}{q_1}{r_1}}^{1-\theta}\|f\|_{\ihsf{s_2}{p_2}{q_2}{r_2}}^\theta.
\end{align*}
Sufficiency of the condition (i) is proved by Cases I and IV.

\noindent \emph{Case V.}
Suppose that $s_* >  s$,  $p_* =p$, and $p_1 \neq p_2$. We may assume that $p_1<p<p_2$. Then since
$$
0<\theta= \frac{1/p_1-1/p}{1/p_1-1/p_2}<1,
$$
there exists a number $\overline{p}_1$, close to $p$, such that
$$
p_1<\overline{p}_1<p<\overline{p}_2:=\frac{p\overline{p}_1}{2\overline{p}_1-p}<p_2,
$$
$$
0<\mu := \frac{1/p_1-1/\overline{p}_1}{1/p_1-1/p_2}<\theta, \quad\theta<\lambda:=2\theta - \mu <1,
$$
$$
(1-\mu)s_1+\mu s_2 -s>0, \quad\text{and}\quad (1-\lambda )s_1 + \lambda s_2 -s >0.
$$
Note that
$$
\frac{2}{p}=\frac{1}{\overline{p}_1}+\frac{1}{\overline{p}_2},
$$
$$
(1-\mu)s_1 + \mu s_2 - s > n \left( \frac{1-\mu}{p_1}+\frac{\mu}{p_2}-\frac{1}{\overline{p}_1}\right)=0,
$$
and
$$
(1-\lambda)s_1 + \lambda s_2 - s > n \left( \frac{1-\lambda}{p_1}+\frac{\lambda}{p_2}-\frac{1}{\overline{p}_2}\right)=0.
$$
Hence by Cases II and IV,
\begin{align*}
\|f\|_{\ihsf{s}{p}{q}{r}} &\les \|f\|_{\ihsf{s}{\overline{p}_1}{\infty}{1}}^{1/2}\|f\|_{\ihsf{s}{\overline{p}_2}{\infty}{1}}^{1/2}\\
	& \les\left(\|f\|_{\ihsf{s_1}{p_1}{q_1}{r_1}}^{1-\mu}\|f\|_{\ihsf{s_2}{p_2}{q_2}{r_2}}^\mu \right)^{1/2}  \left(\|f\|_{\ihsf{s_1}{p_1}{q_1}{r_1}}^{1-\lambda}\|f\|_{\ihsf{s_2}{p_2}{q_2}{r_2}}^\lambda\right)^{1/2} \\
	&= \|f\|_{\ihsf{s_1}{p_1}{q_1}{r_1}}^{1-\theta}\|f\|_{\ihsf{s_2}{p_2}{q_2}{r_2}}^\theta.
\end{align*}

\noindent \emph{Case VI.} Suppose that $s_* -s > n/p_* -n/p >0$. Then
by Theorem \ref{ebd-inhom-TLL} and  Case IV,
$$\|f\|_{\ihsf{s}{p}{q}{r}} \les \|f\|_{\ihsf{s_{*}}{p_{*}}{\infty}{\infty}} \les \|f\|_{\ihsf{s_1}{p_1}{q_1}{r_1}}^{1-\theta}
\|f\|_{\ihsf{s_2}{p_2}{q_2}{r_2}}^{\theta}.
$$

\noindent \emph{Case VII.}
Suppose that $s_* -s  = n/p_* -n/p >0$ and $s_2 -s_1 \neq n/p_2 -n/p_1$.
Then since
$$
 0<\theta =  \frac{s-n/p -s_1 + n/p_1}{s_2 -n/p_2 -s_1 + n/p_1}<1 \quad \mbox{and}\quad \frac{1}{p}< \frac{1-\theta}{p_1}+\frac{\theta}{p_2} \le 1 ,
$$
there exist     numbers  $\overline{p}  >1$ and $\tilde{p} >1$,  very close to $p$,  such that
$$
0<\lambda := \frac{s- n/\overline{p}-s_1 + n/p_1}{s_2 -n/p_2 -s_1 + n/p_1}<\theta
$$
$$
\theta < \mu := \frac{s- n /\tilde{p}-s_1 + n/p_1}{s_2 -n/p_2 -s_1 + n/p_1}<1,
$$
$$
\frac{1}{\overline{p}} < \frac{1-\lambda}{p_1} + \frac{\lambda}{p_2}, \quad\text{and}\quad
\frac{1}{\tilde{p}} < \frac{1-\mu}{p_1} + \frac{\mu}{p_2}.
$$
Note that
\[
(1-\lambda)s_1+ \lambda s_2 -s = n \left( \frac{1-\lambda}{p_1} + \frac{\lambda}{p_2}  - \frac{1}{\overline{p}}  \right)>0
\]
and
\[
(1-\mu)s_1+ \mu s_2 -s = n \left( \frac{1-\mu}{p_1} + \frac{\mu}{p_2}  - \frac{1}{\tilde{p}}  \right)>0 .
\]
Hence by Case IV, we have
\begin{equation*}
 \|f\|_{\ihsf{s}{\overline{p}}{\infty}{r}} \les \|f\|_{\ihsf{s_1}{p_1}{q_1}{r_1}}^{1-\lambda} \|f\|_{\ihsf{s_2}{p_2}{q_2}{r_2}}^\lambda
\quad\mbox{and}\quad
\|f\|_{\ihsf{s}{\tilde{p}}{\infty}{r}} \les \|f\|_{\ihsf{s_1}{p_1}{q_1}{r_1}}^{1-\mu} \|f\|_{\ihsf{s_2}{p_2}{q_2}{r_2}}^\mu.
\end{equation*}
Define
$$
\nu=\frac{\theta -\lambda}{\mu-\lambda}.
$$
Then
$$
0<\nu<1,\quad   \frac{1}{p}=\frac{1-\nu}{\overline{p}}+\frac{\nu}{\tilde{p}}, \quad\text{and}\quad \overline{p}\neq \tilde{p} .
$$
Therefore by Case II, we get
\begin{align*}
\|f\|_{\ihsf{s}{p}{q}{r}} & \les \|f\|_{\ihsf{s}{\overline{p}}{\infty}{r}}^{1-\nu} \|f\|_{\ihsf{s}{\tilde{p}}{\infty}{r}}^{\nu} \\
	& \les \left( \|f\|_{\ihsf{s_1}{p_1}{q_1}{r_1}}^{1-\lambda} \|f\|_{\ihsf{s_2}{p_2}{q_2}{r_2}}^\lambda \right)^{1-\nu} \left( \|f\|_{\ihsf{s_1}{p_1}{q_1}{r_1}}^{1-\mu} \|f\|_{\ihsf{s_2}{p_2}{q_2}{r_2}}^\mu \right)^\nu \\
	& = \|f\|_{\ihsf{s_1}{p_1}{q_1}{r_1}}^{1-\theta} \|f\|_{\ihsf{s_2}{p_2}{q_2}{r_2}}^{\theta}.
\end{align*}
We have  completed the whole proof of the theorem.
\end{proof}

The  following is an immediate consequence of   Remark \ref{rmk-conditions} and Theorem \ref{thm-itp-inhomf}.
\begin{thm}\label{thm-necessity-00}
Let   $s , s_1 , s_2 \in \mathbb{R}$, $1 \le p,   p_1,p_2 < \infty$, $1 \le   q_1 , q_2 , r_1 , r_2 \le \infty$, and $0<\theta<1$.
Then the interpolation inequality
\[
\|f\|_{\ihsf{s}{p}{q_*}{r_*}}\lesssim \|f\|_{\ihsf{s_1}{p_1}{q_1}{r_1}}^{1-\theta}\|f\|_{\ihsf{s_2}{p_2}{q_2}{r_2}}^\theta
\]
holds for all  $f\in F^{s_1,r_1}_{p_1,q_1} \cap F^{s_2,r_2}_{p_2,q_2}$, if and only if $$
s_* -s  \ge \frac{n}{p_*} -\frac{n}{p}\ge 0.
$$
\end{thm}

To further study necessary conditions for stronger interpolation inequalities, we  recall from the theory of real interpolation that if $(X_1,X_2)$ is a compatible couple of quasi-Banach spaces and $X$ is a Banach space with $X_1\cap X_2 \subset X$, then
the interpolation inequality
$
\|f\|_X \les \|f\|_{X_1}^{1-\theta} \|f\|_{X_2}^\theta
$
holds for all $f\in X_1\cap X_2$ if and only if $(X_1,X_2)_{\theta,1}\hookrightarrow X$ (see \cite[Theorem 3.11.4]{Bergh}).

\begin{thm}\label{thm-necessity-0}
Let   $s , s_1 , s_2 \in \mathbb{R}$, $1 < p,p_1,p_2 < \infty$, $1 \le   q_1 , q_2 \le \infty$, and $0<\theta<1$.
Then the interpolation inequality
\begin{equation}\label{thm-necessity-inequality-0}
\|f\|_{F^{s,1}_{p,q_*}}\les \|f\|_{F^{s_1,\infty}_{p_1,q_1}}^{1-\theta}\|f\|_{F^{s_2,\infty}_{p_2,q_2}}^\theta
\end{equation}
holds for all $f\in F^{s_1,\infty}_{p_1,q_1} \cap F^{s_2,\infty}_{p_2,q_2}$, if and only if one of the following conditions is satisfied:
\begin{enumerate}[label = \textup{(\roman*)}]
\item $s_* =s$, $p_*=p$, $s_1\neq s_2$.
\item $s_* >  s$, $s_*-s \ge n/p_*-n/p \ge 0$.
\end{enumerate}
\end{thm}

\begin{proof} By Theorem \ref{thm-itp-inhomf}, it remains to prove the necessity part of the theorem. Suppose   that $\eqref{thm-necessity-inequality-0}$ holds for all $f\in F^{s_1,\infty}_{p_1,q_1} \cap F^{s_2,\infty}_{p_2,q_2}$. Then by Remark \ref{rmk-conditions}, we have
$$
s_* -s  \ge \frac{n}{p_*} -\frac{n}{p}\ge 0.
$$
Hence there are   two  possibilities:
 (a) $s_* =s$, $p_*=p$ and (b) $s_* >  s$, $s_*-s \ge n/p_*-n/p \ge 0$.
Suppose that   $s_1=s_2$. If $p_1\neq p_2$, then    by   Theorem \ref{lem-itp-inhomf}  and the reiteration theorem,
$$
F^{s_*,\infty}_{p_*,1}  \hookrightarrow (F^{s_1,\infty}_{p_1,q_1},F^{s_2,\infty}_{p_2,q_2})_{\theta,1} ,
$$
which holds trivially when $p_1 =p_2$.
Moreover, since $F^{s,1}_{p,q_*}$ is a Banach space, we  have
$$
F^{s_*,\infty}_{p_*,1}  \hookrightarrow (F^{s_1,\infty}_{p_1,q_1}, F^{s_2,\infty}_{p_2,q_2})_{\theta,1}\hookrightarrow F^{s,1}_{p,q_*}.
$$
Hence it follows from Theorem \ref{ebd-inhom-TLL} that
 $s_* >  s$.
By contraposition, we have shown  that if $s_*  =s$, then $s_1 \neq s_2$.
This completes the proof.
\end{proof}

\begin{thm}\label{thm-necessity-01}
Let   $s , s_1 , s_2 \in \mathbb{R}$, $1  <   p < \infty$, $1 < p_1,p_2 < \infty$, $1\le r_1,r_2\le\infty$ and $0<\theta<1$.
If the interpolation inequality
\begin{equation}\label{thm-necessity-inequality-01}
\|f\|_{F^{s,r_*}_{p,1}}\les \|f\|_{F^{s_1,r_1}_{p_1,\infty}}^{1-\theta}\|f\|_{F^{s_2,r_2}_{p_2,\infty}}^\theta
\end{equation}
holds for all $f\in F^{s_1,r_1}_{p_1,\infty} \cap F^{s_2,r_2}_{p_2,\infty}$, then one of the following conditions is satisfied:
\begin{enumerate}[label = \textup{(\roman*)}]
\item $s_* =s$, $p_*=p$, $p_1\neq p_2$, $s_2-s_1\neq n/p_2- n/p_1$.
\item $s_* >  s$, $p_*=p$, $p_1 \neq p_2$.
\item $s_*-s=n/p_* - n/p>0$, $s_2-s_1\neq n/p_2 - n/p_1$.
\item $s_*-s>n/p_*-n/p>0$.
\end{enumerate}
\end{thm}

\begin{proof} Suppose   that $\eqref{thm-necessity-inequality-01}$ holds for all $f\in F^{s_1,r_1}_{p_1,\infty} \cap F^{s_2,r_2}_{p_2,\infty}$, that is, $(F^{s_1,r_1}_{p_1,\infty}, F^{s_2,r_2}_{p_2,\infty})_{\theta,1}\hookrightarrow F^{s,r_*}_{p,1}$. Then by Remark \ref{rmk-conditions}, there are four possibilities:
 (a) $s_* = s$, $p_*=p$, (b) $s_* >  s$, $p_*=p$,
 (c) $s_*-s = n/p_*-n/p>0$, and (d) $s_*-s>n/p_*-n/p>0$.

\noindent \emph{Case I.} Suppose that    $p_1=p_2$. If $s_1 =s_2$, then by Lemma \ref{thm-ebd-FB and BF},
 $$
 B^{s_*,1}_{p_*} \hookrightarrow F^{s_*,1}_{p_* }  \hookrightarrow (F^{s_1,r_1}_{p_1,\infty}, F^{s_2,r_2}_{p_2,\infty})_{\theta,1}\hookrightarrow F^{s,r_*}_{p,1}.
 $$
 If $s_1 \neq s_2$, then by   Theorem \ref{lem-itp-homf-ihomf}, we also have
 $$
 B^{s_*,1}_{p_*}= (F^{s_1,r_1}_{p_1}, F^{s_2,r_2}_{p_2})_{\theta,1}  \hookrightarrow (F^{s_1,r_1}_{p_1,\infty}, F^{s_2,r_2}_{p_2,\infty})_{\theta,1}\hookrightarrow F^{s,r_*}_{p,1}.
 $$
 But it was shown in \cite[Theorem 1.1]{seeger} that $B^{s_*,1}_{p_*}\hookrightarrow \ihsf{s}{p}{1}{r_*}$ if and only if $p_*<p$.  Hence it follows that $p_* < p$.
By contraposition, we deduce   that if $p_* =p$, then $p_1 \neq p_2$.

\noindent \emph{Case II.} Suppose that $s_*-s=n/p_*-n/p$ and $s_2-s_1=n/p_2-n/p_1$.
 By symmetry, we may assume that $s_1\le s_2$.
Then since $s-n/p=s_1-n/p_1=s_2- n/p_2$, it follows from  Theorem \ref{ebd-inhom-TLL} that $\ihsf{s_2}{p_2}{\infty}{1}\hookrightarrow (F^{s_1,r_1}_{p_1,\infty},F^{s_2,r_2}_{p_2,\infty})_{\theta,1}$ but  $\ihsf{s_2}{p_2}{\infty}{1} \nhra \ihsf{s}{p}{1}{r_*}$.
This implies that if $s_*-s=n/p_*-n/p$, then $s_1-s_2 \neq n/p_1-n/p_2$.

Combining the conclusions from Cases I  and II, we   complete the proof of the theorem.
\end{proof}

\begin{thm}\label{thm-necessity}
Let   $s , s_1 , s_2 \in \mathbb{R}$, $1 < p,p_1,p_2 < \infty$, and $0<\theta<1$.
If the interpolation inequality
\begin{equation*}
\|f\|_{F^{s,1}_{p,1}}\les \|f\|_{F^{s_1,\infty}_{p_1,\infty}}^{1-\theta}\|f\|_{F^{s_2,\infty}_{p_2,\infty}}^\theta
\end{equation*}
holds for all $f\in F^{s_1,\infty}_{p_1,\infty} \cap F^{s_2,\infty}_{p_2,\infty}$, then one of the following conditions is satisfied:
\begin{enumerate}[label = \textup{(\roman*)}]
\item $s_* =s$, $p_*=p$, $s_1\neq s_2$, $p_1\neq p_2$, $s_2-s_1\neq n/p_2 - n/p_1$.
\item $s_* >  s$, $p_*=p$, $p_1 \neq p_2$.
\item $s_*-s=n/p_* - n/p>0$, $s_2-s_1\neq n/p_2 - n/p_1$.
\item $s_*-s>n/p_*-n/p>0$.
\end{enumerate}
\end{thm}

\begin{proof} The theorem immediately follows from Theorems \ref{thm-necessity-0} and \ref{thm-necessity-01}.
\end{proof}

The following is  the homogeneous counterpart  of Theorem  \ref{thm-itp-inhomf}.

\begin{thm}\label{thm-itp-homf}
Assume that
\begin{equation*}
 s_* -s = \frac{n}{p_*}-\frac{n}{p} \ge 0 .
\end{equation*}
Assume in addition that $ 1 < p , p_1 , p_2 < \infty $.
Then the interpolation inequality
\[
\|f\|_{\hsf{s}{p}{q}{r}}\lesssim \|f\|_{\hsf{s_1}{p_1}{q_1}{r_1}}^{1-\theta}\|f\|_{\hsf{s_2}{p_2}{q_2}{r_2}}^\theta
\]
holds for all $f \in\hsf{s_1}{p_1}{q_1}{r_1} \cap \hsf{s_2}{p_2}{q_2}{r_2}$, if one of the following conditions is satisfied:
\begin{enumerate}[label = \textup{(\roman*)}]
\item    $q_* \le q$, $r_* \le r$.
\item $s = s_1=s_2$, $p_1 \neq p_2$, $\max (r_1 , r_2 ) \le r$.
\item  $s_* =s$,  $s_1 \neq s_2$,  $q_* \le q$.
\item $s_* >  s$,   $q_* \le q$.
\item  $s_* >  s$, $s_2 -s_1 \neq n/p_2 -n/p_1$.
\end{enumerate}
\end{thm}
\begin{proof}
The proof  of the theorem is exactly the same as that of Theorem \ref{thm-itp-inhomf} except for using Theorem \ref{ebd-hom-TLL} instead of Theorem \ref{ebd-inhom-TLL}.
\end{proof}

The proofs of Theorems \ref{thm-necessity-00}, \ref{thm-necessity-0},  \ref{thm-necessity-01}, and \ref{thm-necessity} can be easily adapted to deduce  their homogeneous counterparts from Remark \ref{rmk-conditions}, Theorems \ref{lem-itp-inhomf}, and \ref{ebd-hom-TLL}.

\begin{thm}\label{thm-necessity-00-hom}
Let   $s , s_1 , s_2 \in \mathbb{R}$, $1  < p,p_1,p_2 < \infty$, $1 \le   q_1 , q_2 , r_1 , r_2 \le \infty$, and $0<\theta<1$.
Then the interpolation inequality
\[
\|f\|_{\hsf{s}{p}{q_*}{r_*}}\lesssim \|f\|_{\hsf{s_1}{p_1}{q_1}{r_1}}^{1-\theta}\|f\|_{\hsf{s_2}{p_2}{q_2}{r_2}}^\theta
\]
holds for all  $f\in \dot F^{s_1,r_1}_{p_1,q_1} \cap \dot F^{s_2,r_2}_{p_2,q_2}$, if and only if $$
s_* -s  = \frac{n}{p_*} -\frac{n}{p}\ge 0.
$$
\end{thm}

\begin{thm}\label{thm-necessity-0-hom}
Let   $s , s_1 , s_2 \in \mathbb{R}$, $1 < p,p_1,p_2 < \infty$, $1 \le   q_1 , q_2 \le \infty$, and $0<\theta<1$.
Then the interpolation inequality
\[
\|f\|_{\dot F^{s,1}_{p,q_*}}\les \|f\|_{\dot F^{s_1,\infty}_{p_1,q_1}}^{1-\theta}\|f\|_{\dot F^{s_2,\infty}_{p_2,q_2}}^\theta
\]
holds for all $f\in \dot F^{s_1,\infty}_{p_1,q_1} \cap \dot F^{s_2,\infty}_{p_2,q_2}$, if and only if one of the following conditions is satisfied:
\begin{enumerate}[label = \textup{(\roman*)}]
\item $s_* =s$, $p_*=p$, $s_1\neq s_2$.
\item   $s_*-s  = n/p_*-n/p > 0$.
\end{enumerate}
\end{thm}

\begin{thm}\label{thm-necessity-hom-01}
Let   $s , s_1 , s_2 \in \mathbb{R}$, $1 < p,p_1,p_2 < \infty$, $1\le r_1,r_2\le\infty$, and $0<\theta<1$.
If the interpolation inequality
$$
\|f\|_{\dot F^{s,r_*}_{p,1}}\les \|f\|_{\dot F^{s_1,r_1}_{p_1,\infty}}^{1-\theta}\|f\|_{\dot F^{s_2,r_2}_{p_2,\infty}}^\theta
$$
holds for all $f\in \dot F^{s_1,r_1}_{p_1,\infty} \cap \dot F^{s_2,r_2}_{p_2,\infty}$, then one of the following conditions is satisfied:
\begin{enumerate}[label = \textup{(\roman*)}]
\item $s_* =s$, $p_*=p$, $p_1\neq p_2$, $s_2-s_1\neq n/p_2 - n/p_1$.
\item $s_*-s=n/p_* - n/p>0$, $s_2-s_1\neq n/p_2 - n/p_1$.
\end{enumerate}
\end{thm}

\begin{thm}\label{thm-necessity-hom}
Let   $s , s_1 , s_2 \in \mathbb{R}$, $1 < p,p_1,p_2 < \infty$, and $0<\theta<1$.
If the interpolation inequality
$$
\|f\|_{\dot F^{s,1}_{p,1}}\les \|f\|_{\dot F^{s_1,\infty}_{p_1,\infty}}^{1-\theta}\|f\|_{\dot F^{s_2,\infty}_{p_2,\infty}}^\theta
$$
holds for all $f\in \dot F^{s_1,\infty}_{p_1,\infty} \cap \dot F^{s_2,\infty}_{p_2,\infty}$, then one of the following conditions is satisfied:
\begin{enumerate}[label = \textup{(\roman*)}]
\item $s_* =s$, $p_*=p$, $s_1\neq s_2$, $p_1\neq p_2$, $s_2-s_1\neq n/p_2 - n/p_1$.
\item $s_*-s=n/p_* - n/p>0$, $s_2-s_1\neq n/p_2 - n/p_1$.
\end{enumerate}
\end{thm}

\section{Interpolation inequalities in Besov-Lorentz spaces}

Using the arguments in the proofs of Theorems \ref{thm-itp-inhomf} and \ref{thm-itp-homf}, we prove interpolation inequalities in Besov-Lorentz spaces.

\begin{thm}\label{thm-itp-inhomb}
Assume that
\[
 s_* -s   \ge   \frac{n}{p_*}-\frac{n}{p} \ge 0 .
\]
 Then the interpolation inequality
\begin{equation*}
\|f\|_{\ihsb{s}{p}{q}{r}}\lesssim \|f\|_{\ihsb{s_1}{p_1}{q_1}{r_1}}^{1-\theta}\|f\|_{\ihsb{s_2}{p_2}{q_2}{r_2}}^\theta
\end{equation*}
holds for all $f \in\ihsb{s_1}{p_1}{q_1}{r_1} \cap \ihsb{s_2}{p_2}{q_2}{r_2}$, if one of the following conditions is satisfied:
\begin{enumerate}[label = \textup{(\roman*)}]
\item    $ q_* \le q$, $r_* \le r$.
\item $p_* =p$, $p_1 \neq p_2$, $r_* \le r$.
\item    $p_* < p$, $r_* \le r$.
\item $s_1\neq s_2$, $p=p_1=p_2$, $\max(q_1,q_2)\le q$.
\item $s_* >  s$, $p_*=p$, $q_*\le q$.
\item $s_* >  s$, $p_* =p$, $p_1 \neq p_2$.
\item  $s_* -s > n/p_* -n/p >0$.
\item  $s_* -s = n/p_* -n/p >0$,  $s_2 -  s_1 \neq  n/p_2 -n/p_1$.
\end{enumerate}
\end{thm}
\begin{proof} Assume that $f\in \ihsb{s_1}{p_1}{q_1}{r_1} \cap \ihsb{s_2}{p_2}{q_2}{r_2}$.

\vspace{0.2cm}

\noindent \emph{Case I.}
Suppose that   $p_* =p$, $q_* \le q$, and $r_* \le r$. Then  by  the same argument as   Case I in the proof of Theorem \ref{thm-itp-inhomf},
\begin{align*}
\|f\|_{\ihsb{s}{p}{q}{r}} &\les   \inhomnormb{s_*}{p}{q}{r}{\psi}{f}\\
&\les \left\| \left\{ 2^{(1-\theta) js_1}|\Delta_j^\psi f|^{1-\theta}\right\}_{j\in\mathbb{N}_0} \right\|_{l^{\frac{r_1}{1-\theta}} ( L^{\frac{p_1}{1-\theta},\frac{q_1}{1-\theta}})} \left\| \left\{ 2^{\theta js_2}|\Delta_j^\psi f|^{\theta}\right\} _{j\in\mathbb{N}_0} \right\|_{l^{\frac{r_2}{\theta}} ( L^{\frac{p_2}{\theta},\frac{q_2}{\theta}}) }\\
	& =  \|f\|^{1-\theta}_{\ihsb{s_1}{p_1}{q_1}{r_1}} \|f\|^\theta_{\ihsb{s_2}{p_2}{q_2}{r_2}}.
\end{align*}

\noindent \emph{Case II.} Suppose that $p_*  =p$, $p_1 \neq p_2$, and $r_* \le r$.
Then by Lemma \ref{interpolation-Lp} and   H\"older's inequality,
\begin{align*}
\|f\|_{\ihsb{s}{p}{q}{r}}  & \les   \left\| \left\{2^{js_*}\|\Delta_j^\psi f\|_{L^{p,q}}  \right\}_{j\in\mathbb{N}_0} \right\|_{l^{r} }\\
	&\les \left\| \left\{ 2^{j(1-\theta)s_1}\|\Delta_j^\psi f\|_{L^{p_1,q_1}}^{1-\theta} \right\}_{j\in\mathbb{N}_0} \right\|_{l^{\frac{r_1}{1-\theta}} } \left\| \left\{  2^{j\theta s_2}\|\Delta_j^\psi f\|_{L^{p_2,q_2}}^{\theta} \right\} _{j\in\mathbb{N}_0} \right\|_{l^{\frac{r_2}{\theta}}  }\\
	&= \|f\|^{1-\theta}_{\ihsb{s_1}{p_1}{q_1}{r_1}} \|f\|^\theta_{\ihsb{s_2}{p_2}{q_2}{r_2}}.
\end{align*}

\noindent \emph{Case III.}
Suppose that  $p_* < p$ and  $r_* \le r$.
Then by Case I, we have
$$\|f\|_{\ihsb{s_*}{p_*}{\infty}{r}} \les  \|f\|_{\ihsb{s_1}{p_1}{q_1}{r_1}}^{1-\theta}\|f\|_{\ihsb{s_2}{p_2}{q_2}{r_2}}^\theta.$$
Since $s_*- s \ge n /p_* -n /p  > 0$, it follows from  Theorem \ref{ebd-inhom-B} and  Remark \ref{rmk-ebd-inhom-B} that
$$
\|f\|_{\ihsb{s}{p}{q}{r}}\les \|f\|_{\ihsb{s_1}{p_1}{q_1}{r_1}}^{1-\theta}\|f\|_{\ihsb{s_2}{p_2}{q_2}{r_2}}^\theta.
$$
Sufficiency of the condition (i) is proved by Cases I and III.

\noindent \emph{Case IV.} If  $s_1\neq s_2$, $p=p_1=p_2$, and $\max(q_1,q_2)\le q$, then by   Lemma \ref{itp-sum},
\begin{align*}
\|f\|_{\ihsb{s}{p}{q}{r}} & \les  \left\| \left\{2^{js_*}\|\Delta_j^\psi f\|_{L^{p,q}}  \right\}_{j\in\mathbb{N}_0} \right\|_{l^{1} } \\
	& \les \left\| \left\{2^{js_1}\|\Delta_j^\psi f\|_{L^{p,q}}  \right\}_{j\in\mathbb{N}_0} \right\|_{l^{\infty} }^{1-\theta} \left\| \left\{2^{js_2}\|\Delta_j^\psi f\|_{L^{p,q}}  \right\}_{j\in\mathbb{N}_0} \right\|_{l^{\infty} }^{\theta}   \\
&\les \|f\|_{\ihsb{s_1}{p_1}{q_1}{r_1}} ^{1-\theta} \|f\|_{\ihsb{s_2}{p_2}{q_2}{r_2}}^\theta.
\end{align*}

\noindent \emph{Case V.} If $s_* >  s$, $p_*=p$, and $q_*\le q$, then by Theorem \ref{ebd-inhom-B} and Case I,
$$
\|f\|_{\ihsb{s}{p}{q}{r}}\les\|f\|_{\ihsb{s_*}{p}{q}{\infty}} \les  \|f\|_{\ihsb{s_1}{p_1}{q_1}{r_1}}^{1-\theta}\|f\|_{\ihsb{s_2}{p_2}{q_2}{r_2}}^\theta  .
$$

\noindent \emph{Case VI.}
 If $s_* >  s$,  $p_* =p$, and $p_1 \neq p_2$, then by Theorem \ref{ebd-inhom-B} and Case II,
$$
\|f\|_{\ihsb{s}{p}{q}{r}}\les \|f\|_{\ihsb{s_*}{p}{q}{\infty}} \les  \|f\|_{\ihsb{s_1}{p_1}{q_1}{r_1}}^{1-\theta}\|f\|_{\ihsb{s_2}{p_2}{q_2}{r_2}}^\theta.
$$

\noindent \emph{Case VII.} If $s_* -s > n/p_* -n/p >0$, then by Theorem \ref{ebd-inhom-B} and Case III,
$$\|f\|_{\ihsb{s}{p}{q}{r}}\les \|f\|_{\ihsb{s_{*}}{p_{*}}{\infty}{\infty}} \les \|f\|_{\ihsb{s_1}{p_1}{q_1}{r_1}}^{1-\theta}\|f\|_{\ihsb{s_2}{p_2}{q_2}{r_2}}^{\theta}.  $$

\noindent \emph{Case VIII.}
Suppose that $s_* -s = n/p_* -n/p >0$ and   $s_2 -  s_1 \neq  n/p_2 -n/p_1$.
Choose a nonzero number $\delta$ such that
$$
s+|\delta| < s_*,
$$
$$
0<\lambda := \frac{s-\delta - n/p-s_1 + n/p_1}{s_2 -n/p_2 -s_1 + n/p_1}<\theta,
$$
$$
\theta < \mu := \frac{s+\delta- n /p-s_1 + n/p_1}{s_2 -n/p_2 -s_1 + n/p_1}<1,
$$
$$
\frac{1}{p} < \frac{1-\lambda}{p_1} + \frac{\lambda}{p_2}, \quad\text{and}\quad
\frac{1}{p} < \frac{1-\mu}{p_1} + \frac{\mu}{p_2}.
$$
Note that
\[
(1-\lambda)s_1+ \lambda s_2 -(s-\delta) = n \left( \frac{1-\lambda}{p_1} + \frac{\lambda}{p_2}  - \frac{1}{p}  \right)>0
\]
and
\[
(1-\mu)s_1+ \mu s_2 -(s+\delta ) = n \left( \frac{1-\mu}{p_1} + \frac{\mu}{p_2}  - \frac{1}{p}  \right)>0 .
\]
Hence by Case III, we have
\begin{equation*}
 \|f\|_{\ihsb{s-\delta}{{p}}{1}{\infty}} \les \|f\|_{\ihsb{s_1}{p_1}{q_1}{r_1}}^{1-\lambda} \|f\|_{\ihsb{s_2}{p_2}{q_2}{r_2}}^\lambda
\quad\mbox{and}\quad
\|f\|_{\ihsb{s+\delta}{{p}}{1}{\infty}} \les \|f\|_{\ihsb{s_1}{p_1}{q_1}{r_1}}^{1-\mu} \|f\|_{\ihsb{s_2}{p_2}{q_2}{r_2}}^\mu.
\end{equation*}
Therefore, by Case IV, we get
\begin{align*}
\|f\|_{\ihsb{s}{p}{q}{r}} & \les \|f\|_{\ihsb{s-\delta}{p}{1}{\infty}}^{1/2} \|f\|_{\ihsb{s+\delta}{p}{1}{\infty}}^{1/2} \\
	& \les \left( \|f\|_{\ihsb{s_1}{p_1}{q_1}{r_1}}^{1-\lambda} \|f\|_{\ihsb{s_2}{p_2}{q_2}{r_2}}^\lambda \right)^{1/2} \left( \|f\|_{\ihsb{s_1}{p_1}{q_1}{r_1}}^{1-\mu} \|f\|_{\ihsb{s_2}{p_2}{q_2}{r_2}}^\mu \right)^{1/2} \\
	& = \|f\|_{\ihsb{s_1}{p_1}{q_1}{r_1}}^{1-\theta} \|f\|_{\ihsb{s_2}{p_2}{q_2}{r_2}}^{\theta},
\end{align*}
which completes the proof.
\end{proof}

\begin{thm}\label{thm-itp-homb}
Assume that
\[
 s_* -s   =   \frac{n}{p_*}-\frac{n}{p} \ge 0 .
\]
Then the interpolation inequality
\begin{equation*}
\|f\|_{\hsb{s}{p}{q}{r}}\lesssim \|f\|_{\hsb{s_1}{p_1}{q_1}{r_1}}^{1-\theta}\|f\|_{\hsb{s_2}{p_2}{q_2}{r_2}}^\theta
\end{equation*}
holds for all $f \in\hsb{s_1}{p_1}{q_1}{r_1} \cap \hsb{s_2}{p_2}{q_2}{r_2}$, if one of the following conditions is satisfied:
\begin{enumerate}[label = \textup{(\roman*)}]
\item    $ q_* \le q$, $r_* \le r$.
\item $s_* =s $, $p_1 \neq p_2$, $r_* \le r$.
\item   $s_* >  s $, $r_* \le r$.
\item    $s_* =s$, $s_1\neq s_2$, $p_1=p_2$, $\max(q_1,q_2)\le q$.
\item  $s_* >  s $,  $s_2 -  s_1 \neq  n/p_2 -n/p_1$.
\end{enumerate}
\end{thm}
\begin{proof}
The  theorem follows by using Theorem \ref{ebd-hom-B} instead of Theorem \ref{ebd-inhom-B} in the proof of Theorem  \ref{thm-itp-inhomb}.
\end{proof}

From    Remark \ref{rmk-conditions}, Theorems \ref{thm-itp-inhomb}, and  \ref{thm-itp-homb}, we immediately obtain:

\begin{thm}\label{thm-necessity-00-b}
The interpolation inequality
\[
\|f\|_{\ihsb{s}{p}{q_*}{r_*}}\lesssim \|f\|_{\ihsb{s_1}{p_1}{q_1}{r_1}}^{1-\theta}\|f\|_{\ihsb{s_2}{p_2}{q_2}{r_2}}^\theta
\]
holds for all  $f\in B^{s_1,r_1}_{p_1,q_1} \cap B^{s_2,r_2}_{p_2,q_2}$, if and only if $$
s_* -s  \ge \frac{n}{p_*} -\frac{n}{p}\ge 0.
$$
\end{thm}

\begin{thm}\label{thm-necessity-00-b-hom}
The interpolation inequality
\[
\|f\|_{\hsb{s}{p}{q_*}{r_*}}\lesssim \|f\|_{\hsb{s_1}{p_1}{q_1}{r_1}}^{1-\theta}\|f\|_{\hsb{s_2}{p_2}{q_2}{r_2}}^\theta
\]
holds for all  $f\in \dot B^{s_1,r_1}_{p_1,q_1} \cap \dot B^{s_2,r_2}_{p_2,q_2}$, if and only if $$
s_* -s  = \frac{n}{p_*} -\frac{n}{p}\ge 0.
$$
\end{thm}

Using the necessity part of Theorem \ref{ebd-inhom-B}, we can obtain several results on necessary conditions for interpolation inequalities in inhomogeneous  Besov-Lorentz spaces.

\begin{thm}\label{thm-necessity-01-b}
Let   $s , s_1 , s_2 \in \mathbb{R}$, $1 <  p <\infty$, $1<p_1,p_2 \le  \infty$, $1\le r_1,r_2\le\infty$ and $0<\theta<1$.
Then the interpolation inequality
\begin{equation}\label{thm-necessity-inequality-01-b}
\|f\|_{B^{s,r_*}_{p,1}}\les \|f\|_{B^{s_1,r_1}_{p_1,\infty}}^{1-\theta}\|f\|_{B^{s_2,r_2}_{p_2,\infty}}^\theta
\end{equation}
holds for all $f\in B^{s_1,r_1}_{p_1,\infty} \cap B^{s_2,r_2}_{p_2,\infty}$, if and only if one of the following conditions is satisfied:
\begin{enumerate}[label = \textup{(\roman*)}]
\item $s_* \ge s$, $p_*=p$, $p_1\neq p_2$.
\item $s_*-s \ge n/p_* - n/p>0$.
\end{enumerate}
\end{thm}

\begin{proof} By Theorem \ref{thm-itp-inhomb}, it suffices to prove the necessity part of the theorem.
Suppose that $\eqref{thm-necessity-inequality-01-b}$ holds for all $f\in B^{s_1,r_1}_{p_1,\infty} \cap B^{s_2,r_2}_{p_2,\infty}$. Then since
$$
s_* -s  \ge \frac{n}{p_*} -\frac{n}{p}\ge 0,
$$
 there are two possibilities:
 (a) $s_* \ge s$, $p_*=p$ and (b) $s_*-s \ge n/p_* - n/p>0$.
 Suppose that $p_1=p_2$. If $s_1\neq s_2$, then by Theorem \ref{lem-itp-homb},
 $$
 B^{s_*,1}_{p_*} \hookrightarrow (B^{s_1,r_1}_{p_1}, B^{s_2,r_2}_{p_2})_{\theta,1}\hookrightarrow (B^{s_1,r_1}_{p_1,\infty}, B^{s_2,r_2}_{p_2,\infty})_{\theta,1}.
 $$
 which holds trivially when $s_1 =s_2$.
 Moreover, since $\ihsb{s}{p}{1}{r_*}$ is a Banach space,
 $$
  (B^{s_1,r_1}_{p_1,\infty}, B^{s_2,r_2}_{p_2,\infty})_{\theta,1}\hookrightarrow \ihsb{s}{p}{1}{r_*}.
  $$
By Theorem \ref{ebd-inhom-B}, we have $p_*<p$. Consequently, $p_*=p$ implies $p_1\neq p_2$. This completes the proof.
\end{proof}

\begin{thm}\label{thm-necessity-0-b}
Let   $s , s_1 , s_2 \in \mathbb{R}$, $1 < p,p_1,p_2 < \infty$, $1 \le   q_1 , q_2 \le \infty$, and $0<\theta<1$.
If the interpolation inequality
\begin{equation}\label{thm-necessity-inequality-0-b}
\|f\|_{B^{s,1}_{p,q_*}}\les \|f\|_{B^{s_1,\infty}_{p_1,q_1}}^{1-\theta}\|f\|_{B^{s_2,\infty}_{p_2,q_2}}^\theta
\end{equation}
holds for all $f\in B^{s_1,\infty}_{p_1,q_1} \cap B^{s_2,\infty}_{p_2,q_2}$, then one of the following conditions is satisfied:
\begin{enumerate}[label = \textup{(\roman*)}]
\item $s_* =s$, $p_*=p$, $s_1\neq s_2$, $s_2-s_1\neq n/p_2-n/p_1$.
\item $s_* > s$, $p_*=p$.
\item $s_*-s=n/p_*-n/p>0$, $s_2-s_1\neq n/p_2-n/p_1$.
\item $s_*-s>n/p_*-n/p>0$.
\end{enumerate}
\end{thm}

\begin{proof} Suppose that $\eqref{thm-necessity-inequality-0-b}$ holds for all $f\in B^{s_1,\infty}_{p_1,q_1} \cap B^{s_2,\infty}_{p_2,q_2}$. Then
there are four  possibilities:
 (a) $s_* =s$, $p_*=p$, (b) $s_*  > s$, $p_*=p$, (c) $s_*-s = n/p_*-n/p > 0$, and (d) $s_*-s > n/p_*-n/p > 0$.

Suppose that   $s_1=s_2$. Then   by Lemma \ref{thm-ebd-FB and BF}, Theorem \ref{lem-itp-inhomf},  and the reiteration theorem,
$$
F^{s_*,\infty}_{p_*,1}  \hookrightarrow (F^{s_1,\infty}_{p_1,q_1},F^{s_2,\infty}_{p_2,q_2})_{\theta,1} \hookrightarrow (B^{s_1,\infty}_{p_1,q_1},B^{s_2,\infty}_{p_2,q_2})_{\theta,1}\hookrightarrow B^{s,1}_{p,q_*}.
$$
Hence it follows from \cite[Theorem 1.2]{seeger} that $s_* > s $. By contraposition, we have shown that if $s_* =s$, then $s_1 \neq s_2$.

Suppose next that $s_*-s=n/p_*-n/p$ and $s_2-s_1=n/p_2-n/p_1$. We may assume that $s_1 \le s_2$. Then by Theorem \ref{ebd-inhom-B},
$$
B^{s_2,\infty}_{p_2, 1} \hookrightarrow(B^{s_1,\infty}_{p_1,q_1}, B^{s_2,\infty}_{p_2,q_2})_{\theta,1}\quad\text{but}\quad B^{s_2,\infty}_{p_2, 1} \nhra B^{s,1}_{p,q_*}.
$$
Therefore, it follows that if $s_*-s=n/p_*-n/p$, then $s_2-s_1\neq n/p_2-n/p_1$. This completes the proof.
\end{proof}

\begin{thm}\label{thm-necessity-b}
Let   $s , s_1 , s_2 \in \mathbb{R}$, $1 < p,p_1,p_2 < \infty$, and $0<\theta<1$.
If the interpolation inequality
\begin{equation*}
\|f\|_{B^{s,1}_{p,1}}\les \|f\|_{B^{s_1,\infty}_{p_1,\infty}}^{1-\theta}\|f\|_{B^{s_2,\infty}_{p_2,\infty}}^\theta
\end{equation*}
holds for all $f\in B^{s_1,\infty}_{p_1,\infty} \cap B^{s_2,\infty}_{p_2,\infty}$, then one of the following conditions is satisfied:
\begin{enumerate}[label = \textup{(\roman*)}]
\item $s_* =s$, $p_*=p$, $s_1\neq s_2$, $p_1\neq p_2$, $s_2-s_1\neq n/p_2 - n/p_1$.
\item $s_* >s$, $p_*=p$, $p_1 \neq p_2$.
\item $s_*-s=n/p_* - n/p>0$, $s_2-s_1\neq n/p_2 - n/p_1$.
\item $s_*-s>n/p_*-n/p>0$.
\end{enumerate}
\end{thm}

\begin{proof} The theorem immediately follows from Theorems \ref{thm-necessity-01-b} and \ref{thm-necessity-0-b}.
\end{proof}

Compared with homogeneous Triebel-Lizorkin-Lorentz spaces, one advantage of Besov-Lorentz spaces is that the interpolation inequalities in Theorems \ref{thm-itp-inhomb} and \ref{thm-itp-homb} do hold  for the limiting case when one of $p$, $p_1$, and $p_2$ is equal to $1$ or $\infty$.
Some conclusions of Theorem \ref{thm-itp-homf} can be extended to  the limiting case when $p=1$, $p_1 =1$, or $ p_2=1$, by using the elementary embedding results in Lemma \ref{thm-ebd-FB and BF}.

\begin{thm}\label{cor0}
Let   $s , s_1 , s_2 \in \mathbb{R}$, $1 \le p,p_1,p_2 < \infty$, $1 \le   q , q_1 , q_2   \le \infty$, and $0<\theta<1$.
Then the interpolation inequality
\[
\|f\|_{\hsf{s}{p}{q}{1}}\lesssim \|f\|_{\hsf{s_1}{p_1}{q_1}{\infty}}^{1-\theta}\|f\|_{\hsf{s_2}{p_2}{q_2}{\infty}}^\theta
\]
holds for all $f \in\hsf{s_1}{p_1}{q_1}{\infty} \cap \hsf{s_2}{p_2}{q_2}{\infty}$, if one of the following conditions is satisfied:
\begin{enumerate}[label = \textup{(\roman*)}]
\item   $s_* =s$, $s_1\neq s_2$, $p=p_1=p_2 =1$, $q=q_1 =q_2 =1$.
\item  $  s_* -s =n/p_* -n/p >0$, $s_2 -s_1 \neq n/p_2 -n/p_1$.
\end{enumerate}
\end{thm}
\begin{proof} The theorem is an immediate consequence of Theorem \ref{thm-itp-homb} and Lemma \ref{thm-ebd-FB and BF}.
\end{proof}

\section{Application to  Gagliardo-Nirenberg inequalities}

Applying interpolation inequalities in the previous sections, we shall extend   Gagliardo-Nirenberg inequalities to  Lorentz spaces, which are  of some interest to theory of partial differential equations (see \cite{McC2} e.g.)

\subsection{Sobolev-Lorentz spaces}

For $s\in\mathbb{R}$, we define $J^s=(I-\Delta)^{s/2}:\scc'\rightarrow\scc'$  via the Fourier transform  by
$$
\widehat{J^s f} =  \left( 1+|\xi|^2 \right)^{s/2} \hat f \quad\mbox{for}\,\, f \in \scc'.
$$
Then for $s\in\mathbb{R}$,
$1 < p < \infty$, and $1 \le q \le \infty$, the Sobolev-Lorentz space $\ihs{s}{p}{q}$ is defined by
$$
\ihs{s}{p}{q}=\{ f \in \tpd : J^s f \in L^{p,q}\}.
$$
The space $\ihs{s}{p}{q}$ equipped with the quasi-norm
$\|f \|_{\ihs{s}{p}{q}} = \| J^s f \|_{L^{p,q}}$ is a  quasi-Banach space. Obviously,  $J^\sigma$ is an isometric isomorphism from $\ihs{s}{p}{q}$ onto $\ihs{s-\sigma}{p}{q}$ for $\sigma\in\mathbb{R}$. The homogeneous Sobolev-Lorentz space $\hs{s}{p}{q}$ is defined by
$$
\hs{s}{p}{q}=\{ f \in \tpd_0 : \Lambda^s f \in L^{p,q}\} ,
$$
where $\Lambda^s=(-\Delta)^{s/2}:\tpd_0\rightarrow\tpd_0$ is defined by
$$
\langle\Lambda^sf,\psi\rangle=  \left\lg f , \left( |\xi|^{s}  \psi\ift \right)\wdh \right\rg \quad\text{for every $f\in\scc_0' , \, \psi\in \scc_0$}.
$$
Note that $\Lambda^\sigma$ is an isometric isomorphism from $\hs{s}{p}{q}$ onto $\hs{s-\sigma}{p}{q}$ for any $\sigma\in\mathbb{R}$.
For $1<p=q<\infty$, we write  $H_{p}^s = \ihs{s}{p}{p}$ and $\dot{H}_p^s =  \hs{s}{p}{p}$.

\begin{thm}\label{SFE2-inhom}
For $s\in\mathbb{R}$, $1<p<\infty$, and $1\le  q \le \infty$, we have
\[
\ihsf{s}{p}{q}{2}= \ihs{s}{p}{q}\quad \text{and} \quad \hsf{s}{p}{q}{2}=  \hs{s}{p}{q}.
\]
\end{thm}

\begin{proof}
It was shown in    \cite[Theorem 2.3.8 and  Section 5.2.3]{Tri1}   that if $s \in\mathbb{R}$,  $1 \le p<\infty$, and $1\le r\le \infty$, then $J^s$ maps $F^{s,r}_p$ isomorphically onto $F^{0,r}_p$ and $\Lambda^s$ maps $\dot F^{s,r}_p$ isomorphically onto $\dot F^{0,r}_p$.
It follows from  the interpolation results in Theorem  \ref{lem-itp-inhomf}  that $J^s$ maps $\ihsf{s}{p}{q}{r}$ isomorphically onto $\ihsf{0}{p}{q}{r}$ and
$\Lambda^s$ maps $\hsf{s}{p}{q}{r}$ isomorphically onto $\hsf{0}{p}{q}{r}$ for $1<r\le\infty$. Hence  by Part (i) of Theorem  \ref{prop-relation}, we complete the proof.
\end{proof}

By virtue of Theorem \ref{SFE2-inhom}, the following embedding results for Sobolev-Lorentz spaces are immediate consequences of Theorems  \ref{ebd-inhom-TLL}  and  \ref{ebd-hom-TLL}.

\begin{thm}\label{embedding-inhom}
Let $s_1,s_2\in\mathbb{R}$, $1< p_1,p_2<\infty$, and $1 \le  q_1,q_2\le \infty$. Then the embedding
$$
\ihs{s_1}{p_1}{q_1} \hookrightarrow \ihs{s_2}{p_2}{q_2}
$$
holds if and only if one of the following conditions is satisfied:
\begin{enumerate}[label = \textup{(\roman*)}]
\item  $s_1\ge s_2$, $p_1=p_2$, $q_1\le q_2$.
\item $s_1 -s_2 = n/p_1- n/p_2 >0 $, $q_1\le q_2$.
\item  $s_1 -s_2 > n/p_1- n/p_2 >0 $.
\end{enumerate}
\end{thm}

\begin{thm}\label{ebd-hom}
Let $s_1,s_2\in\mathbb{R}$, $1< p_1,p_2<\infty$, and $1\le  q_1,q_2\le \infty$. Then the embedding
$$
\hs{s_1}{p_1}{q_1} \hookrightarrow \hs{s_2}{p_2}{q_2}
$$
holds if and only if
$$
s_1-s_2=\frac{n}{p_1}-\frac{n}{p_2}\ge0\quad\text{and}\quad q_1\le q_2.
$$
\end{thm}

For $k\in\mathbb{N}$, $1 \le  p <  \infty$, and $1\le q\le \infty$, we consider
\begin{align*}
W^{k,p,q} = \{f\in \scc' : D^\alpha f \in L^{p,q} \,\,\text{for}\,\, |\alpha|\le k\},
\end{align*}
which is a quasi-Banach space equipped with the quasi-norm
$$
\|f\|_{W^{k,p,q}} = \sum_{|\alpha|\le k} \|D^\alpha f\|_{L^{p,q}}.
$$
The corresponding homogeneous space $\dot W^{k,p,q}$ is defined by
$$\dot W^{k,p,q} = \{ f \in \scc_0' : D^\alpha f \in L^{p,q} \,\,\text{for}\,\, |\alpha|= k\}.$$
More precisely, $\dot W^{k,p,q}$ consists of all $f\in\scc_0'$ having an extension $F\in\scc'$ such that $D^\alpha F\in L^{p,q}$ for all multi-indices $\alpha$ with $|\alpha|=k$. In this case, we define $D^\alpha f= D^\alpha F$ for $|\alpha|=k$ because such an extension $F$ is unique up to polynomials of degree at most $k-1$. Then $\dot W^{k,p,q}$ is a quasi-Banach space equipped with the quasi-norm
$$
\|f\|_{\dot W^{k,p,q}} = \sum_{|\alpha|= k} \|D^\alpha f\|_{L^{p,q}}.
$$
We define $W^{k,p}=W^{k,p,p}$ and $\dot W^{k,p} = \dot W^{k,p,p}$.
Then it is well-known that
$$
W^{k,p} = F^{k,2}_p =H^{k}_{p} \quad\mbox{and}\quad \dot{W}^{k,p}= \dot{F}^{k,2}_p = \dot{H}^{k}_{p}
$$
 for $1<p<\infty$ (see \cite[Subsections 2.5.6 and 5.2.3]{Tri1} e.g.).

\begin{thm}\label{Interpolation result-W}
Let $1<p, p_1 , p_2 <\infty$ and $0<\theta<1$ satisfy
$$
p_1 \neq p_2 \quad \mbox{and}\quad \frac{1}{p}= \frac{1-\theta}{p_1}+ \frac{\theta}{p_2}.
$$
Then for $k \in {\mathbb N}$ and $1 \le q \le \infty$,
$$
(W^{k,p_1},W^{k,p_2})_{\theta,q}=W^{k,p,q} \quad \text{and} \quad (\dot W^{k,p_1},\dot W^{k,p_2})_{\theta,q}=\dot W^{k,p,q}.
$$
\end{thm}

\begin{proof} Since the first interpolation result   is proved in \cite[Theorem 6]{Adams}, it remains to prove the second one.

Let $K(t,f;X_1,X_2)$ denote  the
 $K$-functional for a given compatible couple $(X_1,X_2)$ of two quasi-Banach spaces:
$$
K(t,f;X_1,X_2) = \inf\{ \|f_1\|_{X_1} + t\|f_2\|_{X_2} : f=f_1+f_2,\, f_1\in X_1,\, f_2\in X_2 \}.
$$
Then it is easy to show that if $f \in (\dot W^{k,p_1},\dot W^{k,p_2})_{\theta,q}$, then
$$
\sum_{|\alpha| =k} K(t, D^\alpha f; L^{p_1}, L^{p_2})  \les K(t,f;\dot W^{k,p_1}, W^{k,p_2}) .
$$
 To prove the reverse inequality, suppose that $f\in  \dot W^{k,p,q}$.
Note  that
\[
f   = \frac{1}{(2\pi i)^k} \sum_{|\alpha|=k}\binom{k}{\alpha}\Lambda^{-k}\left[ \left( \frac{\xi^\alpha}{|\xi|^{k}}\widehat{D^\alpha f}\right)   \ift\right] \quad \text{in $\scc_0'$}.
\]
For each multi-index $\alpha$ with $|\alpha|=k$, we  choose $g_1^\alpha \in L^{p_1}$ and $g_2^\alpha \in L^{p_2}$ such that $D^\alpha f= g_1^\alpha + g_2^\alpha$.
Recall  now that if  $|\alpha|=k$, then $\xi^\alpha/|\xi|^k$ is a $L^r$-Fourier multiplier for $1<r<\infty$. Therefore,
\begin{align*}
K(t,f; \dot W^{k,p_1},W^{k,p_2}) &\les \sum_{|\alpha|=k}\left( \left\| \Lambda^{-k}\left( \frac{\xi^\alpha}{|\xi|^{k}}\widehat{g_1^\alpha}\right)\ift\right\|_{\dot W^{k,p_1}} + t   \left\| \Lambda^{-k}\left( \frac{\xi^\alpha}{|\xi|^{k}}\widehat{g_2^\alpha}\right)\ift\right\|_{\dot W^{k,p_2}}\right) \\
	& \sim \sum_{|\alpha|=k} \left(\left\|  \Lambda^{-k} \left( \frac{\xi^\alpha}{|\xi|^{k}}\widehat{g_1^\alpha}\right)\ift\right\|_{\dot H^{k}_{p_1}} + t  \left\|  \Lambda^{-k} \left( \frac{\xi^\alpha}{|\xi|^{k}}\widehat{g_2^\alpha}\right)\ift\right\|_{\dot H^{k}_{p_2}}\right) \\
	& = \sum_{|\alpha|=k} \left( \left\| \left( \frac{\xi^\alpha}{|\xi|^{k}}\widehat{g_1^\alpha}\right)\ift\right\|_{L^{p_1}} + t   \left\| \left( \frac{\xi^\alpha}{|\xi|^{k}}\widehat{g_2^\alpha}\right)\ift\right\|_{L^{p_2}}\right) \\
	& \les \sum_{|\alpha|=k} \left( \|g_1^\alpha \|_{L^{p_1}} + t \|g_2^\alpha\|_{L^{p_2}}\right).
\end{align*}
By the arbitrariness of $g_1^\alpha$ and $g_2^\alpha$ with $D^\alpha f= g_1^\alpha + g_2^\alpha$, we conclude that
$$
K(t,f; \dot W^{k,p_1},W^{k,p_2}) \les \sum_{|\alpha| =k} K(t, D^\alpha f; L^{p_1}, L^{p_2}) .
$$
This completes the proof.
\end{proof}

By Theorems \ref{lem-itp-inhomf}, \ref{SFE2-inhom}, and \ref{Interpolation result-W}, we immediately  obtain

\begin{thm}\label{H equal W}
For any  $k\in\mathbb{N}$, $1<p<\infty$, and $1\le q\le\infty$,
$$
\ihs{k}{p}{q}= W^{k,p,q} \quad \text{and} \quad \hs{k}{p}{q} =\dot W^{k,p,q}.
$$
\end{thm}

The embedding results below will be used   to prove  Gagliardo-Nirenberg inequalities  involving  the critical Lorentz space $L^{1,\infty}$.

\begin{thm}\label{thm-ebd-B}
\begin{enumerate}[label = \textup{(\roman*)}]
\item
If $1\le p \le \infty$, then $ \dot B^{n/p,1}_{p,\infty} \hookrightarrow L^\infty.$
\item If $1\le p < \infty$, then $ L^1 \hookrightarrow \dot B^{-n(1-1/p),\infty}_{p,1}$.
\item If $1 \le p<\infty$,  then $\dot B^{n/p,\infty}_{p,\infty} \hookrightarrow BMO$;
more precisely, for each $f\in \dot B^{n/p,\infty}_{p,\infty}$, there exists  $g\in BMO$, unique up to additive constants, such that
\begin{equation}\label{BMO-equation}
\lg f, \vp \rg = \int g\vp \quad\text{for all $\vp\in \scc_0$}
\quad\mbox{and}\quad
\|g\|_{BMO}\les \|f\|_{\dot B^{n/p,\infty}_{p,\infty}}.
\end{equation}
\end{enumerate}
\end{thm}
\begin{proof}
 (i)  It follows from Theorems  \ref{ebd-hom-B}  and \ref{prop-relation} that $ \dot B^{n/p,1}_{p,\infty} \hookrightarrow \dot B^{0,1}_{\infty} \hookrightarrow L^\infty.$

(ii) By Theorems   \ref{prop-relation} and \ref{ebd-hom-B},  $ L^1 \hookrightarrow \dot B^{0,\infty}_{1}\hookrightarrow \dot B^{-n(1-1/p),\infty}_{p,1}$.

(iii) By Theorem \ref{ebd-hom-B}, it suffices to show that $\dot B^{n/p,\infty}_{p} \hookrightarrow BMO$ for $1<p<\infty$. Suppose that $1<p<\infty$ and $f\in \dot B^{n/p,\infty}_p$. Recall  from the duality theorem in \cite[Chapter 3]{Peetre2}  and Theorem \ref{thm-ebd-FB-2} that
$$
 \left  (\dot B^{-n/p,1}_{p/(p-1)} \right)^* = \dot B^{n/p,\infty}_p \quad\text{and} \quad \scc_0 \subset \dot F^{0,2}_1   \hookrightarrow \dot B^{-n/p,1}_{p/(p-1)}.
 $$
Hence for all $\vp\in\scc_0$, we have
$$
| \lg f, \vp \rg | \les \|f\|_{\dot B^{n/p,\infty}_p} \|\vp\|_{\dot B^{-n/p,1}_{p/(p-1)}} \les \|f\|_{\dot B^{n/p,\infty}_p} \|\vp\|_{\dot F^{0,2}_{1}}.
$$
It is well-known that $\scc_0$ is dense in $\dot F^{0,2}_1$ and $\dot F^{0,2}_1$ is the Hardy space $\mathcal{H}^1$ (see, e.g., \cite[Theorem 2.2.9]{Gra1}). Thus there exists $F\in \left(\mathcal{H}^1\right)^*$ such that
$$
\lg f, \vp \rg = \lg F, \vp \rg \quad\text{for all $\vp\in \scc_0$} \quad \text{and}\quad \|F\|_{(\mathcal{H}^1)^*} \les \|f\|_{\dot B^{n/p,\infty}_p}.
$$
Moreover, since $BMO$ is the dual space of $\mathcal{H}^1$, there exists $g\in BMO$ such that
$$
\int_{\{|x|\le1\}} g =0, \quad \|g\|_{BMO} \les  \|F\|_{(\mathcal{H}^1)^*},
$$
and
\begin{equation}\label{BMO-eq}
\lg F, \vp \rg = \int g\vp
\end{equation}
for all $\vp\in \mathcal{H}^1_0$, where $\mathcal{H}^1_0$ is the space of all finite linear combinations of $L^2$-atoms for $\mathcal{H}^1$.

Now we show that \eqref{BMO-eq}   holds for all $\vp\in\scc_0$. Note that the integral in \eqref{BMO-eq} is well-defined for $\vp\in\scc_0$. Indeed, since
$
\int_{\{|x|\le1\}} g =0,
$
it follows from \cite[Proposition 3.1.5]{Gra1} that
\begin{equation}\label{BMO-eq2}
\int \frac{|g(x)|}{(1+|x|)^{n+1}}dx \les \|g\|_{BMO}.
\end{equation}
Suppose that  $\vp\in\scc_0.$ Then since $\vp \in \dot F^{0,2}_1 = \mathcal{H}^1$, there exists a sequence $\{\vp_N\}$ in $\mathcal{H}^1_0$ such that $\vp_N \rightarrow \vp$ in $\mathcal{H}^1$.  For $M>0$, define
\begin{equation*}
g_M =
\begin{cases}
 \quad M &\quad\text{if}\,\, g > M  \\
 \quad \,g &\quad\text{if}\,\, |g| \le M  \\
 -M &\quad\text{if}\,\, g<-M.
\end{cases}
\end{equation*}
Then since $g_M$ is bounded,
\begin{align*}
\left| \int g_M (\vp_N - \vp) \right| &\les \|g_M\|_{BMO} \|\vp_N - \vp\|_{\mathcal{H}^1} \les \|g\|_{BMO} \|\vp_N - \vp\|_{\mathcal{H}^1}.
\end{align*}
Note that $|g_M (\vp_N-\vp)| \le |g(\vp_N-\vp)| \in L^1$. Thus letting $M\rightarrow \infty$, we obtain
\begin{align*}
\left| \int g (\vp_N - \vp) \right|  \les \|g\|_{BMO} \|\vp_N - \vp\|_{\mathcal{H}^1}.
\end{align*}
Therefore,
$$
\lg F, \vp \rg = \lim_{N\rightarrow \infty} \lg F, \vp_N \rg = \lim_{N\rightarrow \infty} \int g \vp_N = \int g \vp,
$$
which proves   \eqref{BMO-equation}. Finally, the uniqueness of $g$ follows from \eqref{BMO-eq2} (see \cite[p. 243]{Tri1}).
\end{proof}

\begin{thm}\label{thm-ebd-Linf}
If $1<p<\infty$, then
$$
\dot{H}_{p,1}^{n/p} \hookrightarrow\hsf{n/p}{p}{1}{\infty} \hookrightarrow L^\infty, \quad \dot{H}_{p,\infty}^{n/p} \hookrightarrow\hsf{n/p}{p}{\infty}{\infty} \hookrightarrow BMO,
$$
and
$$
 L^1 \hookrightarrow \hsf{-n(1-1/p)}{p}{\infty}{1}  \hookrightarrow \dot{H}^{-n(1-1/p)}_{p,\infty}.
$$
\end{thm}
\begin{proof} By Theorem  \ref{SFE2-inhom},  Lemma \ref{thm-ebd-FB and BF}, and Theorem \ref{thm-ebd-B}, we have
\[
\dot{H}_{p,\infty}^{n/p} \hookrightarrow\hsf{n/p}{p}{\infty}{\infty} \hookrightarrow \hsb{n/p}{p}{\infty}{\infty} \hookrightarrow BMO .
\]
Choose any numbers   $p_1$ and  $p_2$  such that $1<p_1<p<p_2 < \infty$. Then by Theorems  \ref{thm-ebd-B}, \ref{thm-ebd-FB}, and \ref{SFE2-inhom},
\[
\dot{H}_{p,1}^{n/p} \hookrightarrow\hsf{n/p}{p}{1}{\infty}  \hookrightarrow  	  \dot B^{n/p_2 ,1}_{p_2} \hookrightarrow   L^\infty
\]
and
\[
   L^1   \hookrightarrow \dot{B}^{-n(1-1/p_1),\infty}_{p_1,1}\hookrightarrow \hsf{-n(1-1/p)}{p}{\infty}{1}  \hookrightarrow \dot{H}^{-n(1-1/p)}_{p,\infty}. \qedhere
\]
\end{proof}

The corresponding  embedding results also hold for inhomogeneous spaces, proofs  of which are exactly the same as that of Theorem \ref{thm-ebd-Linf} and so omitted.
\begin{thm}\label{thm-ebd-Linf-inhom}
If $1<p<\infty$, then
$$
H_{p,1}^{n/p} \hookrightarrow \ihsf{n/p}{p}{1}{\infty} \hookrightarrow L^\infty, \quad H_{p,\infty}^{n/p} \hookrightarrow\ihsf{n/p}{p}{\infty}{\infty} \hookrightarrow BMO,
$$
and
$$
 L^1 \hookrightarrow \ihsf{-n(1-1/p)}{p}{\infty}{1}\hookrightarrow H^{-n(1-1/p)}_{p,\infty}.
$$
\end{thm}

\subsection{Gagliardo-Nirenberg inequalities in Lorentz spaces}


From Theorems   \ref{thm-itp-inhomf},  \ref{thm-itp-homf}, and \ref{SFE2-inhom}, we immediately obtain the following  interpolation inequalities in Sobolev-Lorentz spaces, which are indeed   Galiardo-Nirenberg inequalities for fractional derivatives in  Lorentz spaces.

\begin{thm}\label{thm-itp-inhomh} Let $s, s_1,s_2\in\mathbb{R}$, $1< p, p_1,p_2<\infty$, and $1\le q, q_1,q_2\le \infty$.
Assume that
\[
 s_* -s \ge  \frac{n}{p_*}-\frac{n}{p} \ge 0 .
\]
Then the interpolation inequality
\begin{equation*}\label{thm-itp-sobolev-inhom}
\|f\|_{\ihs{s}{p}{q}}\lesssim \|f\|_{\ihs{s_1}{p_1}{q_1}}^{1-\theta}\|f\|_{\ihs{s_2}{p_2}{q_2}}^\theta
\end{equation*}
holds for all $f \in\ihs{s_1}{p_1}{q_1} \cap \ihs{s_2}{p_2}{q_2}$, if one of the following conditions is satisfied:
\begin{enumerate}[label = \textup{(\roman*)}]
\item $q_* \le q$.
\item $s_1=s_2$, $p_* =p$, $p_1 \neq p_2$.
\item $s_* >s$, $p_*=p$, $p_1\neq p_2$.
\item  $s_* -s > n/p_* -n/p >0$.
\item  $s_* -s  = n/p_* -n/p >0$, $s_2 -s_1 \neq n/p_2 -n/p_1$.
\end{enumerate}
\end{thm}

\begin{thm}\label{thm-itp-homh} Let $s, s_1,s_2\in\mathbb{R}$, $1< p, p_1,p_2<\infty$, and $1\le q, q_1,q_2\le \infty$.
Assume that
\[
 s_* -s  =  \frac{n}{p_*}-\frac{n}{p} \ge 0 .
\]
 Then the interpolation inequality
\begin{equation*}
\|f\|_{\hs{s}{p}{q}}\lesssim \|f\|_{\hs{s_1}{p_1}{q_1}}^{1-\theta}\|f\|_{\hs{s_2}{p_2}{q_2}}^\theta
\end{equation*}
holds for all $f \in\hs{s_1}{p_1}{q_1} \cap \hs{s_2}{p_2}{q_2}$, if one of the following conditions is satisfied:
\begin{enumerate}[label = \textup{(\roman*)}]
\item $q_* \le q$.
\item $s =s_1=s_2 $, $p_1 \neq p_2$.
\item  $ s_* >s$, $s_2 -s_1 \neq n/p_2 -n/p_1$.
\end{enumerate}
\end{thm}

Next,   we  show that   Gagliardo-Nirenberg  inequalities in Lorentz spaces    may hold even for    the limiting case when some exponent is equal to $1$ or  $\infty$.


\begin{thm}\label{cor-BL1}
Let $s > 0$, $1 <  p  \le  \infty$, $1 \le  p_1 , p_2 \le \infty$,  and $0<\theta<1$ satisfy
$$
\frac{1}{p}=\frac{1-\theta}{p_1}+\theta\left(\frac{1}{p_2}-\frac{s}{n}\right)
\quad\text{and}\quad\frac{1}{p_1} \neq  \frac{1}{p_2}-\frac{s}{n}.
$$
Assume also that
  $  q =1$ if $p<\infty$ and   $q=\infty$ if $p=\infty$, and that
  $  q_1 = \infty$ if $ p_1 >1$ and     $q_1 =1$ if $p_1 =1$.
Then for each  $f \in L^{p_1 , q_1} \cap \hsb{s}{p_2}{\infty}{\infty}$, we have
\[
\|f  -c  \|_{L^{p,q}}\lesssim  \|f\|_{L^{p_1,q_1}}^{1-\theta}\|f\|_{\hsb{s}{p_2}{\infty}{\infty}}^\theta ,
\]
where  $c = c_{p_1} (f)$ is some constant with $c_{p_1} (f) =0 $ if $p_1 < \infty$.
\end{thm}

\begin{proof} Suppose that
$f \in L^{p_1 , q_1} \cap \hsb{s}{p_2}{\infty}{\infty}$.
Then by Theorems \ref{thm-itp-homb} and \ref{prop-relation}, there is a constant $c  $ such that
\[
\|f  -c  \|_{L^{p,q}} \les \|f\|_{\dot{B}_{p,q}^{0,1}} \lesssim  \|f\|_{L^{p_1,q_1}}^{1-\theta}\|f\|_{\hsb{s}{p_2}{\infty}{\infty}}^\theta .
\]
It remains to show that if $p_1 < \infty$, then we can take $c=0$. It is clear that if $p <\infty$ and $p_1 < \infty$, then $c  $ must be zero.

Suppose that $p=\infty$ and $p_1 <\infty$.
Choosing any    $\psi  \in C_c^\infty(\mathbb{R}^n)$   such that $\psi=1$ on $\{|\xi|\le1\}$ and $\psi=0$ on $\{|\xi|\ge3/2\}$, we define $\vp_j (\cdot) = \psi (2^{-j}\cdot)-\psi (2^{-j+1}\cdot)$ on $\mathbb{R}^n$ for each $j\in\mathbb{Z}$.
Then  since  $\vp=\{\vp_j\}_{j\in\mathbb{Z}}$  is  a sequence in $C_c^\infty(\mathbb{R}^n)$ satisfying \eqref{cond1-hom}, \eqref{cond2-hom}, and \eqref{cond3-hom}, we have
\[
\sum_{j \in \mathbb{Z}} \|\Delta_j^\vp f \|_{L^{\infty}} \les \|f\|_{\dot{B}_{\infty}^{0,1}} <  \infty ,
\]
which implies that the series $\sum_{j \in \mathbb{Z}} \Delta_j^\vp f$ converges in $L^{\infty}$ to a function   $F$ satisfying $\|F\|_{L^{\infty}} \le   \|f\|_{\dot{B}^{0,1}_{\infty}}$. On the other hand, since  $f \in   L^{p_1 , q_1} \cap L^{\infty}$, it can be easily deduced that the series $\sum_{j \in \mathbb{Z}} \Delta_j^\vp f$ indeed converges to $f$ in the sense of distributions. Hence it follows that $F = f$ identically on $\mathbb{R}^n$. This completes the proof.
\end{proof}

Recall       from Lemma  \ref{thm-ebd-FB and BF}, Theorems  \ref{prop-relation}, and \ref{SFE2-inhom}   that
  $\hsf{s}{p}{\infty}{\infty}  \hookrightarrow \hsb{s}{p}{\infty}{\infty}$ for  $1 \le  p< \infty$ and     $\dot{H}^{s}_{p  ,\infty}  \hookrightarrow \hsf{s}{p}{\infty}{\infty} $ if    $1 <  p< \infty$.
Therefore, from Theorem \ref{cor-BL1}, we can derive various interpolation inequalities which   generalize Gagliardo-Nirenberg inequalities. Some of them are listed below.

\begin{eg} The conditions of  Theorem \ref{cor-BL1} are all satisfied when $1 \le  p_1=p_2 <   p <\infty$,  $0< sp_1 <n$, and $1/p_1 -1/p =\theta s/n$. Hence it follows from Theorem \ref{cor-BL1} that  if $0< sp_1 < n$, $1 \le p_1 < p< np_1/(n-sp_1 )$, and $\theta = n(1/p_1 -1/p)/s$, then
\[
\|f\|_{L^{p,1}}  \lesssim \|f\|_{L^{p_1,q_1}}^{1-\theta}\|f\|_{\hsb{s}{p_1}{\infty}{\infty}}^\theta
 \les \|f\|_{L^{p_1,q_1}}^{1-\theta}\|f\|_{\hsf{s}{p_1}{\infty}{\infty}}^{\theta}
 ,
\]
where $q_1 =\infty$ if $p_1 >1$ and $q_1 =1$ if $p_1 =1$. Moreover, if $p_1 >1$, then
\[
\|f\|_{L^{p,1}} \les \|f\|_{L^{p_1,\infty}}^{1-\theta}\|\Lambda^s f\|_{L^{p_1 ,\infty}}^{\theta} \quad\mbox{for all}\,\, f \in H^{s}_{p_1,\infty} .
\]
\end{eg}

\begin{eg} The conditions of  Theorem \ref{cor-BL1} are all satisfied when $1< p < \infty$ $p_1 =\infty$, $1 \le  p_2   <\infty$,  $0<sp_2<n$, and $1/p  =\theta (1/p_2 -s/n )$. Hence,    if $1 \le p_2 < \infty$ and   $0< 1/p < 1/p_2 -s/n$, then
\[
\|f -c  \|_{L^{p,1}}   \les \|f\|_{L^{\infty}}^{1-\theta}\|f\|_{\hsb{s}{p_2}{\infty}{\infty}}^{\theta} \les \|f\|_{L^{\infty}}^{1-\theta}\|f\|_{\hsf{s}{p_2}{\infty}{\infty}}^{\theta}
\]
for some constant $c $,
where $0< \theta <1$ is defined by $1/p  =\theta (1/p_2 -s/n )$.
\end{eg}

\begin{eg} Let $1 \le p_1 <p < \infty$, $1<p_2 < \infty$, and $s =n/p_2 $. Then taking  $\theta=1-p_1/p$ in   Theorem \ref{cor-BL1}, we have
\begin{align*}
\|f\|_{L^{p,1}}&\lesssim    \|f\|_{L^{p_1, q_1}}^{p_1 /p}\|f\|_{\hsb{n/p_2}{p_2}{\infty}{\infty}}^{1-p_1/p}  \\
&\lesssim    \|f\|_{L^{p_1,q_1}}^{p_1/p}\|f\|_{\hsf{n/p_2}{p_2}{\infty}{\infty}}^{1-p_1/p} \\
&\lesssim    \|f\|_{L^{p_1, q_1}}^{p_1 /p}\|\Lambda^{n/p_2} f\|_{L^{p_2 ,\infty}}^{1-p_1/p}  ,
\end{align*}
where $q_1 =\infty$ if $p_1 >1$ and $q_1 =1$ if $p_1 =1$.
In particular, if   $k \in  \mathbb{N}$,  $1 \le k < n$, and $n/k < p<\infty$, then
$$
\|f\|_{L^{p,1}}\les \|f\|_{L^{n/k,\infty}}^{n/kp} \|  f\|_{\dot{W}^{k, n/k,\infty}}^{1-n/kp} \quad\mbox{for all}\,\, f \in W^{k,n/k,\infty} ,
$$
which refines the famous Ladyzhenskaya inequality in \cite{lad}.
\end{eg}

\begin{eg}
Let  $s >0$ and $1< p  <\infty$ satisfy $sp >n$. Then by   Theorem \ref{cor-BL1},
\begin{align*}
\|f\|_{L^\infty}&\lesssim   \|f\|_{L^{p,\infty}}^{1-n/sp}\|f\|_{\hsb{s}{p}{\infty}{\infty}}^{n/sp} \quad\quad\mbox{for all}\,\, f \in B^{s,\infty}_{p,\infty}  \\
&\lesssim    \|f\|_{L^{p,\infty}}^{1-n/sp}\|f\|_{\hsf{s}{p}{\infty}{\infty}}^{n/sp} \quad\quad\mbox{for all}\,\, f \in F^{s,\infty}_{p,\infty} \\
&\lesssim   \|f\|_{L^{p,\infty}}^{1-n/sp} \| \Lambda^s f\|_{L^{p,\infty}}^{n/sp} \quad\mbox{for all}\,\, f \in H^{s}_{p,\infty} .
\end{align*}
By   Theorem \ref{cor-BL1}, we also have
$$
\|f\|_{L^\infty}\les \|f\|_{L^{1}}^{1-\theta} \| \Lambda^s f\|_{L^{p,\infty}}^{\theta} \quad\mbox{for all}\,\, f \in L^1 \cap \dot{H}^{s}_{p,\infty} ,
$$
where $0< \theta <1$ is defined by $\theta (s/n -1/p) = 1-\theta $.
\end{eg}

\begin{eg}
Let $s>0$ and $1< p<\infty$. Then by   Theorem \ref{cor-BL1}, we have
$$
\|f\|_{L^{p,1}}\les \|f\|_{L^{1}}^{1-\theta} \| \Lambda^s f\|_{L^{p,\infty}}^{\theta} \quad\mbox{for all}\,\, f \in L^1 \cap \dot{H}^{s}_{p,\infty}  ,
$$
where $0< \theta <1$ is defined by $\theta s/n = (1-\theta )(1-1/p)$. In particular, taking $s=1$ and  $p=2$, we have
\[
\|f\|_{L^{2,1}} \les \|f\|_{L^{1}}^{2/(n+2)} \|\nabla f\|_{L^{2,\infty}}^{n/(n+2)}\quad\mbox{for all}\,\, f \in L^{1} \cap \dot{W}^{1,2, \infty},
\]
which refines Nash's inequality in \cite{nash}.
\end{eg}

The $L^1$-norms in the above examples can be replaced by the weaker $L^{1,\infty}$-quasinorm, under some additional assumption for the case when $sp_2 >n$.

\begin{thm}\label{cor2}
Let $s > 0$, $1 <  p   \le \infty$, $1 \le     p_1 \le  \infty$, $1< p_2 <  \infty$,  and $0<\theta<1$ satisfy
$$
\frac{1}{p}=\frac{1-\theta}{p_1}+\theta\left(\frac{1}{p_2}-\frac{s}{n}\right)
\quad\text{and}\quad\frac{1}{p_1} \neq  \frac{1}{p_2}-\frac{s}{n}.
$$
Assume also that   $  q =1$ if $p<\infty$ and   $q=\infty$ if $p=\infty$.
\begin{enumerate}[label=\textup{(\roman*)}]
\item If $ p_1 >1$ or $sp_2 \le n$, then
$$
\|f -c  \|_{L^{p,q}}\les \|f\|_{L^{p_1,\infty}}^{1-\theta} \| \Lambda^s f\|_{L^{p_2 ,\infty}}^{\theta} \quad\mbox{for all}\,\, f\in L^{p_1, \infty} \cap \dot{H}^{s}_{p_2 ,\infty} ,
$$
where  $c = c_{p_1} (f)$ is some constant with $c_{p_1} (f) =0 $ if $p_1 < \infty$.
\item If $p_1 =1$ and $sp_2 >n$, then
$$
\|f   \|_{L^{p,q}}\les \|f\|_{L^{1,\infty}}^{1-\theta} \| \Lambda^s f\|_{L^{p_2 ,\infty}}^{\theta} \quad\mbox{for all}\,\, f\in L^{1, \infty} \cap H^{s}_{p_2 ,\infty} .
$$
\end{enumerate}
\end{thm}
\begin{proof}
\noindent \emph{Case I.} If $p_1 >1$, then the result is a  special case of  Theorem \ref{cor-BL1}.

\noindent \emph{Case II.} Suppose that $p_1 =1$, $sp_2 \le n$, and $f\in L^{1, \infty} \cap \dot{H}^{s}_{p_2 ,\infty}$. If  $sp_2 < n$, then by Lemma \ref{interpolation-Lp} and Theorem \ref{ebd-hom},
$$
\|f\|_{L^{p,1}} \les \|f\|_{L^{1,\infty}}^{1-\theta}\|f\|_{L^{np_2 /(n-sp_2),\infty}}^\theta \les \|f\|_{L^{1,\infty}}^{1-\theta}\|\Lambda^s f \|_{L^{p_2 ,\infty}}^\theta.
$$
Suppose next that $sp_2 =n$.
Then it follows from  Theorem \ref{thm-ebd-Linf} that  $\hs{s}{p_2}{\infty}  \hookrightarrow BMO.$ Moreover, it was shown by Hanks  \cite{hanks} that
$(L^{1,\infty}, BMO)_{1-1/p,1}=L^{p,1}$.
Consequently,
$$
\|f\|_{L^{p,1}}\les \|f\|_{L^{1,\infty}}^{1-\theta}\|f\|_{BMO}^{\theta} \les \|f\|_{L^{1,\infty}}^{1-\theta} \|\Lambda^s f\|_{L^{p_2,\infty}}^\theta.
$$

\noindent \emph{Case III.} Suppose that  $p_1 =1$,  $sp_2 >n$,  and $f\in L^{1, \infty} \cap H^{s}_{p_2 ,\infty}$. Then it follows from Theorems \ref{embedding-inhom} and   \ref{thm-ebd-Linf-inhom} that
$f\in L^{1 ,\infty}\cap L^\infty \subset L^{p,q} $.
Since
\[
\frac{1}{p_2}-\frac{s}{n} < \frac{1}{p}< 1 \quad\mbox{and}\quad 0< \theta = \frac{1 -1/p }{1-1/p_2 +s/n } <1,
\]
there exists  a number $\overline{p} >1 $, close to $1$, such that
\[
\frac{1}{p_2}-\frac{s}{n} < \frac{1}{p} < \frac{1}{\overline{p}}
\quad\mbox{and}\quad
0< \lambda := \frac{1/\overline{p} -1/p  }{1/\overline{p}-1/p_2  +s/n  } < 1 .
\]
Define
\[
\mu = \frac{1-1/\overline{p}}{1-1/p} .
\]
Then since $1/ \overline{p} = (1-\mu)/1 +\mu/ p$ and $ 0< \mu < 1$,
it follows from Lemma \ref{interpolation-Lp} that
\[
\|f\|_{L^{\overline{p},\infty}} \les  \|f\|_{L^{1,\infty}}^{1-\mu} \|f\|_{L^{p,\infty}}^\mu .
\]
Note also that
$$
\frac{1}{p}=\frac{1-\lambda}{\overline{p}}+\lambda\left(\frac{1}{p_2}-\frac{s}{n}\right)
\quad\text{and}\quad\frac{1}{\overline{p}} \neq  \frac{1}{p_2}-\frac{s}{n}.
$$
Hence by Case I,
\begin{align*}
\|f\|_{L^{p,q}} &\les \|f\|_{L^{\overline{p},\infty}}^{1-\lambda}\|\Lambda^s f \|_{L^{p_2,\infty}}^\lambda \\
& \les  \left( \|f\|_{L^{1,\infty}}^{1-\mu} \|f\|_{L^{p,\infty}}^\mu \right)^{1-\lambda}\|\Lambda^s f \|_{L^{p_2,\infty}}^\lambda \\
& \le C  \left( \|f\|_{L^{1,\infty}}^{(1-\mu)(1-\lambda)} \|\Lambda^s f \|_{L^{p_2,\infty}}^{\lambda} \right)^{\frac{ 1}{1-\mu (1-\lambda)}} + \frac{1}{2} \|f\|_{L^{p,q}}
\end{align*}
for some constant $C>0$. Therefore, noting that
\[
 \frac{1-\mu(1-\lambda)}{\lambda}= \frac{1-\mu}{\lambda}+ \mu = \frac{1}{\theta} ,
\]
we  have
\[
\|f\|_{L^{p,q}}
  \les    \|f\|_{L^{1,\infty}}^{\frac{(1-\mu)(1-\lambda)}{1-\mu (1-\lambda)}} \|\Lambda^s f \|_{L^{p_2,\infty}}^{\frac{ \lambda}{1-\mu (1-\lambda)}}= \|f\|_{L^{1,\infty}}^{1-\theta}\|\Lambda^s f \|_{L^{p_2,\infty}}^\theta  . \qedhere
\]
\end{proof}

From Theorem \ref{cor2}, we  derive the following generalized Gagliardo-Nirenberg inequalities involving the $L^{1,\infty}$-quasinorm.

\begin{eg}
\begin{enumerate}[label=\textup{(\roman*)}]
\item If $1<p , p_2 < \infty$, then
$$
\|f\|_{L^{p,1}} \lesssim     \|f\|_{L^{1, \infty}}^{1 /p}\|\Lambda^{n/p_2} f\|_{L^{p_2 ,\infty}}^{1- 1/p} \quad\mbox{for all}\,\, f \in L^{1,\infty} \cap \dot H^{n/p_2}_{p_2 ,\infty}.
$$
\item If $s>0$ and $1< p<\infty$  satisfy $sp < n$, then
$$
\|f\|_{L^{p,1}}\les \|f\|_{L^{1,\infty}}^{1-\theta} \| \Lambda^s f\|_{L^{p,\infty}}^{\theta} \quad\mbox{for all}\,\, f \in L^{1,\infty} \cap \dot{H}^{s}_{p,\infty} ,
$$
where $0< \theta <1$ is defined by $\theta s/n = (1-\theta )(1-1/p)$.
\item If $s>0$ and $1< p<\infty$ satisfy $sp>n$, then
$$
\|f\|_{L^\infty}\les \|f\|_{L^{1,\infty}}^{1-\theta} \| \Lambda^s f\|_{L^{p,\infty}}^{\theta} \quad\mbox{for all}\,\, f \in L^{1,\infty} \cap H^{s}_{p,\infty},
$$
where $0< \theta <1$ is defined by $\theta (s/n -1/p) = 1-\theta $.
\end{enumerate}
\end{eg}

\section*{Appendix}

\subsection*{A.1. Proof of Theorem \ref{lem-itp-inhomf}}
Adapting the arguments in \cite{Bui}, we  prove Theorem \ref{lem-itp-inhomf} only for the  homogeneous case. The inhomogeneous case can be proved by the same argument.

The following maximal inequality is due to Fefferman and Stein \cite{FS}.

\begin{thm}\label{FS}
Let $1<p<\infty$ and $1<r\le\infty$. Then for every sequence $\{f_j\}_{j\in\Gamma}$ in $L^{p}(l^r)$,
$$\| \{Mf_j\}_{j\in\Gamma}\|_{L^{p}(l^r)} \les \|\{f_j\}_{j\in\Gamma}\|_{L^{p}(l^r)},$$
where $Mg$ denotes the Hardy-Littlewood maximal function of a function $g$ on $\rn$.
\end{thm}

The quasi-norm $\|\cdot\|_{\hsf{s}{p}{q}{r}}$ of $\hsf{s}{p}{q}{r}$ was defined in terms of a sequence $\{\vp_j\}_{j\in\mathbb{Z}}$ of functions in  $C_c^\infty(\rn)$ satisfying \eqref{cond1-hom}, \eqref{cond2-hom}, and \eqref{cond3-hom}.
But if $\psi \in C_c ^\infty(\rn)$ satisfies $\supp \psi \subset \{1/2 \le |\xi|\le 2\}$ and $\psi >0$ on $\{3/5 \le |\xi|\le 5/3\}$, then the quasi-norm   $\|\cdot\|_{\hsf{s}{p}{q}{r}}$ is equivalent to  the quasi-norm corresponding to $\{\psi_j\}_{j\in\mathbb{Z}}$, where $\psi_j(\xi)=\psi(2^{-j}\xi)$.
The following is taken from \cite[Lemma 6.9]{Fra1} (see also \cite[Exercise 1.1.5]{Gra1}).

\begin{lem}\label{lem-function}
Suppose that  $\psi \in C_c ^\infty(\rn)$ satisfies $\supp \psi \subset \{1/2 \le |\xi| \le 2 \}$ and $\psi >0$ on $\{3/5 \le |\xi| \le  5/3\}$. Then there exists $\vp\in C_c^\infty(\rn)$ such that $\supp \vp \subset \{1/2 \le  |\xi| \le 2\}$, $\vp>0$ on $\{ 3/5 \le |\xi| \le  5/3 \}$, and
$$
\sum_{j\in\mathbb{Z}} \vp( 2^{-j} \xi)\psi( 2^{-j}\xi)=1\quad\text{on }\,\rn\backslash\{0\}.
$$
\end{lem}

\begin{proof}[Proof of Theorem \ref{lem-itp-inhomf} for the homogeneous case]
Let $\psi$ and $\vp$ be two functions in $C_c^\infty(\rn)$ satisfying all the properties of Lemma \ref{lem-function}.
Define $\psi_j(\xi)= \psi(2^{-j}\xi)$ and $\vp_j(\xi)= \vp(2^{-j}\xi)$ for $j\in\mathbb{Z}$.

It is quite easy to show that
$$
K(t, \{2^{js}\Delta_j^\vp f\}_{j\in\mathbb{Z}}; L^{p_1}(l^r), L^{p_2}(l^r)) \les K(t, f ; \dot F ^{s,r}_{p_1}, \dot F^{s,r}_{p_2}) .
$$
Hence to complete the proof, it suffices to show that
$$
K(t, f ; \dot F ^{s,r}_{p_1}, \dot F^{s,r}_{p_2}) \les K (t, \{2^{js}\Delta_j^\vp f\}_{j\in\mathbb{Z}}; L^{p_1}(l^r), L^{p_2}(l^r))\quad\text{for  $1<r\le\infty$}.
$$
Suppose that $1<r \le \infty$ and  $f\in \hsf{s}{p}{q}{r}$. Then by Lemma \ref{interpolation-Lp},
$$
\{ 2^{js} \Delta_j^\vp f\}_{j\in\mathbb{Z}} \in L^{p,q}(l^r) = (L^{p_1}(l^r), L^{p_2}(l^r))_{\theta,q}.
$$
Let $\{g_j\}_{j\in\mathbb{Z}}\in L^{p_1}(l^r)$ and $\{h_j\}_{j\in\mathbb{Z}} \in L^{p_2}(l^r)$ be chosen so that $2^{js}\Delta_j^\vp f = g_j + h_j$ for each $j\in\mathbb{Z}$. Define
$$
f_1 = \sum_{j\in\mathbb{Z}} 2^{-js} \psi_j\ift \ast g_j \quad\text{and}\quad f_2= \sum_{j\in\mathbb{Z}} 2^{-js}\psi_j\ift \ast h_j.
$$
Then since $\sum_{j\in\mathbb{Z}} \vp_j\psi_j =1$ on $\rn\backslash\{0\}$,
$$
f = \sum_{j\in\mathbb{Z}} ( \psi_j \vp_j \hat{f}\,)\ift = \sum_{j\in\mathbb{Z}} \psi_j\ift \ast \Delta_j^\vp f  = f_1 + f_2.
$$
Note that
$$
\Delta_j^\vp f_1 = \sum_{k\in\mathbb{Z}} 2^{-ks}( \vp_j \psi_k \hat{g_k})\ift  = \sum_{l=-1}^{1} 2^{-(j-l)s}  (\vp_j \psi_{j-l})\ift \ast g_{j-l} .
$$
Moreover, if $j\in\mathbb{Z}$ and $-1 \le l \le 1$, then
$$
(\vp_j \psi_{j-l})\ift (x) = 2^{jn}[\vp(\cdot)\psi(2^l \cdot)]\ift (2^j x)
$$
so that
\begin{align*}
\left| (\vp_j \psi_{j-l})\ift \ast g_{j-l}(x) \right|  &\le Mg_{j-l}(x)\int_{\mathbb{R}^n} \left|K_l(y)\right|dy \les Mg_{j-l}(x),
\end{align*}
where $K_l$ is a decreasing majorant of $\vp(\cdot)\psi(2^l\cdot)$
(see \cite[Corollary 2.1.12]{Gra2} for more details).
Hence applying Theorem \ref{FS}, we get
\begin{align*}
\left\| \left\{2^{js}\Delta_j^\vp f_1 \right\}_{j\in\mathbb{Z}}\right\|_{L^{p_1}(l^r)} &\les \sum_{l=-1}^{1} \left\| \left\{M{g}_{j-l}\right\}_{j\in\mathbb{Z}}\right\|_{L^{p_1}(l^r)} \les \left\| \left\{g_j\right\}_{j\in\mathbb{Z}}\right\|_{L^{p_1}(l^r)}.
\end{align*}
Similarly,
$$
\left \|  \{ 2^{js}\Delta_j^\vp f_2  \}_{j\in\mathbb{Z}} \right\|_{L^{p_2}(l^r)}  \les \left \| \{ h_j\} _{j\in\mathbb{Z}}\right \|_{L^{p_2}(l^r)}.
$$
Combining all the estimates, we get
\begin{align*}
 K(t,f;\dot F ^{s,r}_{p_1}, \dot F^{s,r}_{p_2}) &\le \| f_1 \|_{\dot F^{s,r}_{p_1}} + t \|f_2 \|_{\dot F^{s,r}_{p_2}} \\
 	&\les \| \{ g_j\} _{j\in\mathbb{Z}} \|_{L^{p_1}(l^r)} +  t \|  \{h_j\} _{j\in\mathbb{Z}} \|_{L^{p_2}(l^r)}.
\end{align*}
By the arbitrariness of $\{ g_j\} _{j\in\mathbb{Z}}$ and $ \{h_j\} _{j\in\mathbb{Z}}$, we conclude that
$$
K(t, f ; \dot F ^{s,r}_{p_1}, \dot F^{s,r}_{p_2}) \les K(t, \{2^{js}\Delta_j^\vp f\}_{j\in\mathbb{Z}}; L^{p_1}(l^r), L^{p_2}(l^r)).
$$
This completes the proof of Theorem \ref{lem-itp-inhomf}.
\end{proof}

\subsection*{A.2. Proof of Theorem  \ref{prop-relation}}

\begin{proof}[Proof of Theorem \ref{prop-relation}]
(i) It  was proved in  \cite[Theorem 2.5.6]{Tri1} and \cite[Theorem 6.1.2]{Gra2}, for example, that $ F_{p}^{0,2} = \dot{F}_{p}^{0,2}=L^{p}$ for $1<p<\infty$. Hence it follows from  Lemma \ref{interpolation-Lp} and Theorem  \ref{lem-itp-inhomf}  that $\ihsf{0}{p}{q}{2} = \hsf{0}{p}{q}{2}=L^{p,q}$  for $1<p<\infty$ and $1 \le q \le \infty$.

(ii)  We choose  $\psi  \in C_c^\infty(\mathbb{R}^n)$   such that $\psi=1$ on $\{|\xi|\le1\}$ and $\psi=0$ on $\{|\xi|\ge3/2\}$. For each $j\in\mathbb{Z}$, define $\vp_j (\cdot) = \psi (2^{-j}\cdot)-\psi (2^{-j+1}\cdot)$ on $\mathbb{R}^n.$
Then   $\vp=\{\vp_j\}_{j\in\mathbb{Z}}$  is  a sequence in $C_c^\infty(\mathbb{R}^n)$ satisfying \eqref{cond1-hom}, \eqref{cond2-hom}, and \eqref{cond3-hom}.
Next, we define $\eta_0 = \psi$ and $\eta_j = \vp_j$ for $j \ge 1$. Then  $\eta=\{\eta_j\}_{j\ge0}$ satisfies \eqref{cond1-inhom}, \eqref{cond2-inhom}, and \eqref{cond3-inhom}.
Note that $\vp_j = \vp_j \eta_0$ for $j <0 $. Moreover, since
$\vp_j\ift(x)=  2^{jn}\vp_0\ift (2^jx)$,
it follows that
\[
\| \vp_j\ift\|_{L^{1}}  \les  \int_{\mathbb{R}^n} |\psi \ift(x)| dx \les 1 \quad\text{for }j\in \mathbb{Z}.
\]
On the other hand, since $ \Delta_j^\vp f   = \vp_j\ift \ast f$, it follows from Young's convolution inequality that each $\Delta_j^\vp$ is a bounded linear operator on $L^t$ for every $1 \le t \le \infty$. Hence by real interpolation, we deduce that if $1<p<\infty$   or if $1\le p=q \le \infty$, then
$$
\| \Delta_j^\vp f\|_{L^{p,q}}   \les \|\vp_j\ift \|_{L^1} \|f\|_{L^{p,q}} \les \|f\|_{L^{p,q}} \quad\mbox{for all}\,\, j \in \mathbb{Z}.
$$
For $j < 0$, we also have
\begin{equation*}
\|\Delta_j^\vp f \|_{L^{p,q}} = \|  (\vp_j \eta_0 \hat{f}\,)\ift  \|_{L^{p,q}} \les \|(\eta_0 \hat{f}\,)\ift\|_{L^{p,q}} =  \| \Delta_0^\eta f \|_{L^{p,q}}  \les \|f\|_{L^{p,q}}.
\end{equation*}

Suppose now that $1<p<\infty$   or if $  p=q =1$.

We first  show that $\dot{F}^{0,1}_{p,q} \hookrightarrow L^{p,q}$. Suppose that $f \in \dot{F}^{0,1}_{p,q}$. Then since  $\sum_{j \in \mathbb{Z}}|\Delta_j^\vp f | \in L^{p,q}$, the series $\sum_{j \in \mathbb{Z}} \Delta_j^\vp f$ converges in $L^{p,q}$ to a function, denoted by  $F$. It is trivial that  $\|F\|_{L^{p,q}} \le   \|f\|_{\dot{F}^{0,1}_{p,q}}$.
Moreover, since $L^{p,q} \subset L^1 + L^\infty \subset \scc'$ and   $\sum_{j \in \mathbb{Z}} \Delta_j^\vp f     = f$ in $\scc_0'$, it follows that $F$ is a unique extension of $f$ in $L^{p,q}$. This proves that $\dot{F}^{0,1}_{p,q} \hookrightarrow L^{p,q}$.   The embedding $F^{0,1}_{p,q} \hookrightarrow L^{p,q}$ can be proved similarly.

To show that $\ihsf{s}{p}{q}{r}=L^{p,q}\cap \hsf{s}{p}{q}{r}$, suppose that $f \in \ihsf{s}{p}{q}{r}$. Then since $s>0$ and $L^{p,q} (l^r )  \hookrightarrow L^{p,q} (l^\infty)$, we have
$$
\|f\|_{L^{p,q}} \les\left\|\sum_{j=0}^\infty  |\Delta_j^\eta f |  \right\|_{L^{p,q}}  \les \left\| \sup_{j\in\mathbb{N}_0} 2^{js}| \Delta_j^\eta f | \right\|_{L^{p,q}} \les  \|f\|_{  F^{s,r}_{p,q}} .
$$
Moreover, since $L^{p,q} (l^1)  \hookrightarrow L^{p,q} (l^r)$, $L^{p,q}$ is normable, and $s>0$, it follows that
\begin{align*}
\|f\|_{\hsf{s}{p}{q}{r}}	&\sim  \left \| \{ 2^{js} \lpo{\vp}{f} \}_{j\ge0} \right\|_{L^{p,q} (l^r)} +  \left \| \{ 2^{js} \lpo{\vp}{f}\}_{j<0} \right\|_{L^{p,q} (l^r)}\\
	& \les  \left \| \{ 2^{js} \lpo{\vp}{f} \}_{j\ge0} \right\|_{L^{p,q} (l^r)} + \sum_{j<0} 2^{js} \| \lpo{\vp}{f} \|_{L^{p,q}}  \\
	& \les  \left \| \{ 2^{js} \Delta_j^\eta f \}_{j \ge 0} \right\|_{L^{p,q} (l^r)} +  \|  f  \|_{L^{p,q}}
	   \les   \|f\|_{\ihsf{s}{p}{q}{r}}.
\end{align*}
Conversely, if $f \in L^{p,q}\cap \hsf{s}{p}{q}{r}$, then
\[
\|f\|_{\ihsf{s}{p}{q}{r}}	 \sim  \left \| \{ 2^{js} \Delta_j^\eta f \}_{j > 0} \right\|_{L^{p,q} (l^r)} +  \| \Delta_0^\eta f  \|_{L^{p,q}}  \les \|f\|_{\hsf{s}{p}{q}{r}}    +  \|  f  \|_{L^{p,q}} .
\]

(iii) Suppose that    $1<p<\infty$ or $1 \le p=q \le \infty$. Then for all $f\in L^{p,q}$, we have
\begin{align*}
\|f\|_{B^{0,\infty}_{p,q}}+ \|f\|_{\hsb{0}{p}{q}{\infty}} = \sup_{j\in\mathbb{N}_0} \|\Delta_j^\eta f\|_{L^{p,q}} +	 \sup_{j\in\mathbb{Z}} \|\Delta_j^\psi f\|_{L^{p,q}}  \les \|f\|_{L^{p,q}} ,
\end{align*}
which implies that $ L^{p,q} \hookrightarrow B^{0,\infty}_{p,q} $ and $L^{p,q} \hookrightarrow \dot B^{0,\infty}_{p,q}$.  Suppose that $f \in \dot{B}^{0,1}_{p,q}$. Then since  $\sum_{j \in \mathbb{Z}} \|\Delta_j^\vp f \|_{L^{p,q}}< \infty$ and   $L^{p,q}$ is normable,  the series $\sum_{j \in \mathbb{Z}} \Delta_j^\vp f$ converges in $L^{p,q}$ to a function   $F$. Moreover,  $F$ is an extension of $f$ in $L^{p,q}$ satisfying $\|F\|_{L^{p,q}} \le   \|f\|_{\dot{B}^{0,1}_{p,q}}$.
It is also easy to show  that    $\|g\|_{L^{p,q}} \les \|g\|_{  B^{0,1}_{p,q}}$ for all $g\in   B^{0,1}_{p,q}$.
Finally, the proof of  (ii) can be easily adapted to prove that $\ihsb{s}{p}{q}{r}=L^{p,q}\cap \hsb{s}{p}{q}{r}$.
\end{proof}

\subsection*{A.3. Proofs of Theorems  \ref{ebd-hom-TLL} and  \ref{ebd-hom-B}}

\begin{proof}[Proof of Theorem   \ref{ebd-hom-TLL}]
We first prove sufficiency of the conditions (i) and (ii) for the embedding $\hsf{s_1}{p_1}{q_1}{r_1} \hookrightarrow \hsf{s_2}{p_2}{q_2}{r_2}$. Sufficiency of (i) immediately follows from the embedding result \eqref{eq-ebd}.
On the other hand, it follows from Theorems \ref{thm-ebd-FB-2} that if  $s_1-s_2=n/p_1-n/p_2>0$, then $\dot F^{s_1,\infty}_{p_1}\hookrightarrow \dot F^{s_2,1}_{p_2}$. Hence by real interpolation (Theorem \ref{lem-itp-inhomf}), we deduce that if (ii) holds, then
$\hsf{s_1}{p_1}{q_1}{r_1} \hookrightarrow \hsf{s_1}{p_1}{q_2}{\infty} \hookrightarrow \hsf{s_2}{p_2}{q_2}{1} \hookrightarrow \hsf{s_2}{p_2}{q_2}{r_2}.
$

To prove the necessity, we assume that $\hsf{s_1}{p_1}{q_1}{r_1} \hookrightarrow \hsf{s_2}{p_2}{q_2}{r_2}$.
First of all, by a dilation argument, we easily obtain
$$
s_1- \frac{n}{p_1} = s_2 - \frac{n}{p_2}.
$$
Next, let $0<\varepsilon<1/10$ be fixed and choose $\vp \in C_c^\infty (\mathbb{R}^n)$ such that $\supp \vp \subset \{1/2+\varepsilon \le  |\xi| \le 2-\varepsilon\}$, $\vp>0$ on $\{3/5\le |\xi|\le 5/3\}$, and $\vp=1$ on $\{1-\varepsilon \le |\xi|\le 1+\varepsilon\}$. Define $\vp_j(\xi)=\vp(2^{-j}\xi)$ for each $j\in\mathbb{Z}$. Then for $s\in\mathbb{R}$, $1\le p <\infty$, and $1\le q\le\infty$,
$$
\|f\|_{\dot F^{s,r}_{p,q}} \sim \left\| \left\{ 2^{js}\Delta_j^\vp f \right\}_{j\in\mathbb{Z}}\right\|_{L^{p,q}(l^r)}.
$$
Choose any $\psi\in C_c^\infty(\mathbb{R}^n)$ with $\supp \psi \subset \{|\xi|\le\varepsilon\}.$ Then for all $\xi\in\rn$,
\begin{equation*}
\vp_j(\xi)\psi(\xi-2^{k}e_1)=
\begin{cases}
 \psi(\xi-2^j e_1)&\text{if}\,\, j=k\ge 1,\\
 	\,\,0\quad&\text{if} \,\,k\ge1, j\neq k.\\
 \end{cases}
\end{equation*}
Therefore,  if $f\in \scc$ is given by
$$
 \hat{f}(x)=\sum_{k=1}^N a_k \psi(\xi -2^k e_1)
 $$
 for some complex numbers $a_1,\dots, a_N$, then
 \begin{equation*}
\|f\|_{\dot F^{s,r}_{p,q}} \sim \| \{ 2^{js}a_j\}_{j=1}^N \|_{l^r}\|\psi \ift\|_{L^{p,q}}.
\end{equation*}
Since $\hsf{s_1}{p_1}{q_1}{r_1} \hookrightarrow \hsf{s_2}{p_2}{q_2}{r_2}$, we have
$$
\| \{ 2^{js_2}a_j\}_{j=1}^N \|_{l^{r_2}}\|\psi \ift\|_{L^{p_2,q_2}} \les \| \{ 2^{js_1}a_j\}_{j=1}^N \|_{l^{r_1}}\|\psi \ift\|_{L^{p_1,q_1}}
$$
for any   $a_1,\dots, a_N$. Hence it follows that (a) $s_1\ge s_2$  and (b) if $s_1 = s_2$, then $r_1 \le r_2$.

Now, fixing any number $\alpha$ with
$0< \alpha < n \left(1 -1/ p_1 \right)$,
we define
$$
\overline{p}=\frac{np_1}{n+\alpha p_1}.
$$
Then since $1< \overline{p}< p_1$ and $ \alpha= n/ \overline{p} - n/p_1 >0$,
it follows from the sufficiency part of the theorem that
\begin{equation*}\label{proof-lem-ebd}
\hsf{s_1+\alpha}{\overline{p}}{q_1}{\infty} \hookrightarrow \hsf{s_1}{p_1}{q_1}{r_1}\hookrightarrow \hsf{s_2}{p_2}{q_2}{r_1} \hookrightarrow \hsf{s_2}{p_2}{q_2}{\infty}.
\end{equation*}
Recall from the proof of Theorem \ref{SFE2-inhom} that
$\Lambda^\sigma$ maps $\hsf{s}{p}{q}{r}$ isomorphically onto $\hsf{s-\sigma}{p}{q}{r}$ for $s , \sigma \in \mathbb R$,  $1<r\le\infty$, $1 < p< \infty$, and $1 \le q \le \infty$.  Hence by  Theorem   \ref{prop-relation},
$$
\ihsf{s_1+\alpha+\beta}{\overline{p}}{q_1}{\infty} \hookrightarrow \hsf{s_1+\alpha+\beta}{\overline{p}}{q_1}{\infty} \hookrightarrow \hsf{s_2+\beta}{p_2}{q_2}{\infty}
$$
for every $\beta \in {\mathbb R}$ with $s_1 +\alpha +\beta >0 $.
Choose   any $\beta$ with  $\beta >-s_2$. Then since
$$
s_1+ \alpha+\beta   > s_1 +\alpha -s_2 =  \frac{n}{\overline{p}}-\frac{n}{p_2}   >0  \quad\mbox{and}\quad s_2 +\beta >0,
$$
it follows from Theorems \ref{ebd-inhom-TLL} and \ref{prop-relation} that
\begin{align*}
\ihsf{s_1+\alpha+\beta}{\overline{p}}{q_1}{\infty}  \hookrightarrow  \hsf{s_2+\beta}{p_2}{q_2}{\infty} \cap \ihsf{0}{p_2}{q_2}{2}     = \ihsf{s_2+\beta}{p_2}{q_2}{\infty}.
\end{align*}
Therefore, by the necessity part of Theorem \ref{ebd-inhom-TLL}, we deduce that $q_1 \le q_2$.
\end{proof}

To prove Theorem  \ref{ebd-hom-B}, we need the following form of the Bernstein inequality.

\begin{lem}\label{lem-hom-B}
Suppose that  $K \subset{\mathbb R}^n$ is compact, $1 \le p <  \infty$, and  $ 1 \le q \le \infty$. Then
$$
\|f\|_{L^\infty} \les d^{n/p} \|f\|_{L^{p,q}}
$$
for all $f\in L^{p,q}\cap \scc'$ with $\supp \hat{f} \subset K$, where $d$ is the diameter of $K$.
\end{lem}
\begin{proof}
By a simple scaling argument, we may assume that $d=1$. Suppose that $f \in L^{p,q}\cap \scc'$ and $\supp \hat{f} \subset K$. Then choosing $\psi \in \scc$ such that $\hat{\psi}=1$ on $\{\xi \in \rn: |\xi - \eta| \le 1 \text{ for some }\eta \in K\}$, we have $f =  f \ast \psi$ on $\rn$.
Choose a function $\vp\in\scc$ such that $\vp(0)=1$ and $\supp\hat{\vp}\subset\{|\xi|\le1\}$. For $0<\delta<1$, define
$$f_\delta(x)=\vp(\delta x)f(x).$$
Then by Lemma \ref{Holder-Lorentz},
\begin{align*}
|f_\delta(x)| & = \left| \int f_\delta(y) \psi (x-y)dy\right|  \les \| f_\delta\|_{L^{\infty}}^{1/2} \| \psi\|_{L^{2p/(2p-1),2q/(2q-1)}} \| f_\delta\|_{L^{p,q}}^{1/2},
\end{align*}
from which it follows that
\begin{equation*}
\| f_\delta\|_{L^{\infty}}  \les \|f_\delta\|_{L^{p,q}} \les \|f\|_{L^{p,q}}.
\end{equation*}
Letting $\delta \downarrow 0$, we complete the proof.
\end{proof}

\begin{proof}[Proof of Theorem  \ref{ebd-hom-B}]
It immediately follows from \eqref{eq-ebd} that $\hsb{s_1}{p_1}{q_1}{r_1} \hookrightarrow \hsb{s_2}{p_2}{q_2}{r_2} $ when (i) holds.
\noindent Suppose that (ii) holds and $f \in \hsb{s_1}{p_1}{q_1}{r_1}$. Then since the diameter of $\supp \widehat{\Delta_j^\vp f}$ is at most $2^{j+1}$, it follows from Lemma \ref{lem-hom-B} that
$$
\|\Delta_j^\vp f \|_{L^{\infty}} \les 2^{jn/p_1}\|\Delta_j^\vp f\|_{L^{p_1,q_1}} \quad\mbox{for}\,\, j\in\mathbb{Z}.
$$
Hence by Lemma \ref{interpolation-Lp},  we have
\begin{align*}
\|\Delta_j^\vp f \|_{L^{p_2,q_2}} & \les  \|\Delta_j^\vp f \|_{L^{p_1,q_1}}^{p_1/p_2} \|\Delta_j^\vp f \|_{L^{\infty}}^{1-p_1/p_2} \\
	&\les  2^{jn(1/p_1-1/p_2)}  \|\Delta_j^\vp f \|_{L^{p_1,q_1}}=  2^{j(s_1-s_2)}  \|\Delta_j^\vp f \|_{L^{p_1,q_1}},
\end{align*}
from which  it follows that  $\|f\|_{\hsb{s_2}{p_2}{q_2}{r_2}} \les \|f\|_{\hsb{s_1}{p_1}{q_1}{r_1}}$.
\end{proof}

\end{document}